\documentclass[11pt,a4paper]{amsart}
\usepackage{amsfonts}
\usepackage{amscd}
\usepackage{hyperref}
\usepackage{amsmath}
\usepackage{amssymb}

\newtheorem{theorem}{Theorem}
\theoremstyle{plain}

\numberwithin{equation}{section}
\makeatletter
 
\def\diagram{\m@th\leftwidth=\z@ \rightwidth=\z@ \topheight=\z@
\botheight=\z@ \setbox\@picbox\hbox\bgroup}
 
\def\enddiagram{\egroup\wd\@picbox\rightwidth\unitlength
\ht\@picbox\topheight\unitlength \dp\@picbox\botheight\unitlength
\hskip\leftwidth\unitlength\box\@picbox}
 
\def\bfig{\begin{diagram}}
\def\efig{\end{diagram}}
\newcount\wideness \newcount\leftwidth \newcount\rightwidth
\newcount\highness \newcount\topheight \newcount\botheight
 
\def\ratchet#1#2{\ifnum#1<#2 \global #1=#2 \fi}
 
\def\putbox(#1,#2)#3{%
\horsize{\wideness}{#3} \divide\wideness by 2
{\advance\wideness by #1 \ratchet{\rightwidth}{\wideness}}
{\advance\wideness by -#1 \ratchet{\leftwidth}{\wideness}}
\vertsize{\highness}{#3} \divide\highness by 2
{\advance\highness by #2 \ratchet{\topheight}{\highness}}
{\advance\highness by -#2 \ratchet{\botheight}{\highness}}
\put(#1,#2){\makebox(0,0){$#3$}}}
 
\def\putlbox(#1,#2)#3{%
\horsize{\wideness}{#3}
{\advance\wideness by #1 \ratchet{\rightwidth}{\wideness}}
{\ratchet{\leftwidth}{-#1}}
\vertsize{\highness}{#3} \divide\highness by 2
{\advance\highness by #2 \ratchet{\topheight}{\highness}}
{\advance\highness by -#2 \ratchet{\botheight}{\highness}}
\put(#1,#2){\makebox(0,0)[l]{$#3$}}}
 
\def\putrbox(#1,#2)#3{%
\horsize{\wideness}{#3}
{\ratchet{\rightwidth}{#1}}
{\advance\wideness by -#1 \ratchet{\leftwidth}{\wideness}}
\vertsize{\highness}{#3} \divide\highness by 2
{\advance\highness by #2 \ratchet{\topheight}{\highness}}
{\advance\highness by -#2 \ratchet{\botheight}{\highness}}
\put(#1,#2){\makebox(0,0)[r]{$#3$}}}

\def\adjust[#1]{} 
 
\newcount \coefa
\newcount \coefb
\newcount \coefc
\newcount\tempcounta
\newcount\tempcountb
\newcount\tempcountc
\newcount\tempcountd
\newcount\xext
\newcount\yext
\newcount\xoff
\newcount\yoff
\newcount\gap%
\newcount\arrowtypea
\newcount\arrowtypeb
\newcount\arrowtypec
\newcount\arrowtyped
\newcount\arrowtypee
\newcount\height
\newcount\width
\newcount\xpos
\newcount\ypos
\newcount\run
\newcount\rise
\newcount\arrowlength
\newcount\halflength
\newcount\arrowtype
\newdimen\tempdimen
\newdimen\xlen
\newdimen\ylen
\newsavebox{\tempboxa}%
\newsavebox{\tempboxb}%
\newsavebox{\tempboxc}%
 
\newdimen\w@dth
 
\def\setw@dth#1#2{\setbox\z@\hbox{\m@th$#1$}\w@dth=\wd\z@
\setbox\@ne\hbox{\m@th$#2$}\ifnum\w@dth<\wd\@ne \w@dth=\wd\@ne \fi
\advance\w@dth by 1.2em}
 
\def\t@^#1_#2{\allowbreak\def\n@one{#1}\def\n@two{#2}\mathrel
{\setw@dth{#1}{#2}
\mathop{\hbox to \w@dth{\rightarrowfill}}\limits
\ifx\n@one\empty\else ^{\box\z@}\fi
\ifx\n@two\empty\else _{\box\@ne}\fi}}
\def\t@@^#1{\@ifnextchar_{\t@^{#1}}{\t@^{#1}_{}}}
\def\to{\@ifnextchar^{\t@@}{\t@@^{}}}
 
\def\t@left^#1_#2{\def\n@one{#1}\def\n@two{#2}\mathrel{\setw@dth{#1}{#2}
\mathop{\hbox to \w@dth{\leftarrowfill}}\limits
\ifx\n@one\empty\else ^{\box\z@}\fi
\ifx\n@two\empty\else _{\box\@ne}\fi}}
\def\t@@left^#1{\@ifnextchar_{\t@left^{#1}}{\t@left^{#1}_{}}}
\def\toleft{\@ifnextchar^{\t@@left}{\t@@left^{}}}
 
\def\two@^#1_#2{\allowbreak
\def\n@one{#1}\def\n@two{#2}\mathrel{\setw@dth{#1}{#2}
\mathop{\vcenter{\lineskip\z@\baselineskip\z@
                 \hbox to \w@dth{\rightarrowfill}%
                 \hbox to \w@dth{\rightarrowfill}}%
       }\limits
\ifx\n@one\empty\else ^{\box\z@}\fi
\ifx\n@two\empty\else _{\box\@ne}\fi}}
\def\tw@@^#1{\@ifnextchar _{\two@^{#1}}{\two@^{#1}_{}}}
\def\two{\@ifnextchar ^{\tw@@}{\tw@@^{}}}
 
\def\tofr@^#1_#2{\def\n@one{#1}\def\n@two{#2}\mathrel{\setw@dth{#1}{#2}
\mathop{\vcenter{\hbox to \w@dth{\rightarrowfill}\kern-1.7ex
                 \hbox to \w@dth{\leftarrowfill}}%
       }\limits
\ifx\n@one\empty\else ^{\box\z@}\fi
\ifx\n@two\empty\else _{\box\@ne}\fi}}
\def\t@fr@^#1{\@ifnextchar_ {\tofr@^{#1}}{\tofr@^{#1}_{}}}
\def\tofro{\@ifnextchar^ {\t@fr@}{\t@fr@^{}}}

\def\mon{\mathop{\m@th\hbox to
      14.6\P@{\lasyb\char'51\hskip-2.1\P@$\arrext$\hss
$\mathord\rightarrow$}}\limits} 
\def\leftmono{\mathrel{\m@th\hbox to
14.6\P@{$\mathord\leftarrow$\hss$\arrext$\hskip-2.1\P@\lasyb\char'50%
}}\limits} 
\mathchardef\arrext="0200       

\def\settypes(#1,#2,#3){\arrowtypea#1 \arrowtypeb#2 \arrowtypec#3}
\def\settoheight#1#2{\setbox\@tempboxa\hbox{#2}#1\ht\@tempboxa\relax}%
\def\settodepth#1#2{\setbox\@tempboxa\hbox{#2}#1\dp\@tempboxa\relax}%
\def\settokens`#1`#2`#3`#4`{%
     \def\tokena{#1}\def\tokenb{#2}\def\tokenc{#3}\def\tokend{#4}}
\def\setsqparms[#1`#2`#3`#4;#5`#6]{%
\arrowtypea #1
\arrowtypeb #2
\arrowtypec #3
\arrowtyped #4
\width #5
\height #6
}
\def\setpos(#1,#2){\xpos=#1 \ypos#2}

\def\settriparms[#1`#2`#3;#4]{\settripairparms[#1`#2`#3`1`1;#4]}%
 
\def\settripairparms[#1`#2`#3`#4`#5;#6]{%
\arrowtypea #1
\arrowtypeb #2
\arrowtypec #3
\arrowtyped #4
\arrowtypee #5
\width #6
\height #6
}
 
\def\resetparms{\settripairparms[1`1`1`1`1;500]\width 500}
 
\resetparms
 
\def\mvector(#1,#2)#3{
\put(0,0){\vector(#1,#2){#3}}%
\put(0,0){\vector(#1,#2){26}}%
}
\def\evector(#1,#2)#3{{
\arrowlength #3
\put(0,0){\vector(#1,#2){\arrowlength}}%
\advance \arrowlength by-30
\put(0,0){\vector(#1,#2){\arrowlength}}%
}}
 
\def\horsize#1#2{%
\settowidth{\tempdimen}{$#2$}%
#1=\tempdimen
\divide #1 by\unitlength
}
 
\def\vertsize#1#2{%
\settoheight{\tempdimen}{$#2$}%
#1=\tempdimen
\settodepth{\tempdimen}{$#2$}%
\advance #1 by\tempdimen
\divide #1 by\unitlength
}
 
\def\putvector(#1,#2)(#3,#4)#5#6{{%
\ifnum3<\arrowtype
\putdashvector(#1,#2)(#3,#4)#5\arrowtype
\else
\ifnum\arrowtype<-3
\putdashvector(#1,#2)(#3,#4)#5\arrowtype
\else
\xpos=#1
\ypos=#2
\run=#3
\rise=#4
\arrowlength=#5
\ifnum \arrowtype<0
    \ifnum \run=0
        \advance \ypos by-\arrowlength
    \else
        \tempcounta \arrowlength
        \multiply \tempcounta by\rise
        \divide \tempcounta by\run
        \ifnum\run>0
            \advance \xpos by\arrowlength
            \advance \ypos by\tempcounta
        \else
            \advance \xpos by-\arrowlength
            \advance \ypos by-\tempcounta
        \fi
    \fi
    \multiply \arrowtype by-1
    \multiply \rise by-1
    \multiply \run by-1
\fi
\ifcase \arrowtype
\or \put(\xpos,\ypos){\vector(\run,\rise){\arrowlength}}%
\or \put(\xpos,\ypos){\mvector(\run,\rise)\arrowlength}%
\or \put(\xpos,\ypos){\evector(\run,\rise){\arrowlength}}%
\fi\fi\fi
}}
 
\def\putsplitvector(#1,#2)#3#4{
\xpos #1
\ypos #2
\arrowtype #4
\halflength #3
\arrowlength #3
\gap 140
\advance \halflength by-\gap
\divide \halflength by2
\ifnum\arrowtype>0
   \ifcase \arrowtype
   \or \put(\xpos,\ypos){\line(0,-1){\halflength}}%
       \advance\ypos by-\halflength
       \advance\ypos by-\gap
       \put(\xpos,\ypos){\vector(0,-1){\halflength}}%
   \or \put(\xpos,\ypos){\line(0,-1)\halflength}%
       \put(\xpos,\ypos){\vector(0,-1)3}%
       \advance\ypos by-\halflength
       \advance\ypos by-\gap
       \put(\xpos,\ypos){\vector(0,-1){\halflength}}%
   \or \put(\xpos,\ypos){\line(0,-1)\halflength}%
       \advance\ypos by-\halflength
       \advance\ypos by-\gap
       \put(\xpos,\ypos){\evector(0,-1){\halflength}}%
   \fi
\else \arrowtype=-\arrowtype
   \ifcase\arrowtype
   \or \advance \ypos by-\arrowlength
       \put(\xpos,\ypos){\line(0,1){\halflength}}%
       \advance\ypos by\halflength
       \advance\ypos by\gap
       \put(\xpos,\ypos){\vector(0,1){\halflength}}%
   \or \advance \ypos by-\arrowlength
       \put(\xpos,\ypos){\line(0,1)\halflength}%
       \put(\xpos,\ypos){\vector(0,1)3}%
       \advance\ypos by\halflength
       \advance\ypos by\gap
       \put(\xpos,\ypos){\vector(0,1){\halflength}}%
   \or \advance \ypos by-\arrowlength
       \put(\xpos,\ypos){\line(0,1)\halflength}%
       \advance\ypos by\halflength
       \advance\ypos by\gap
       \put(\xpos,\ypos){\evector(0,1){\halflength}}%
   \fi
\fi
}
 
\def\putmorphism(#1)(#2,#3)[#4`#5`#6]#7#8#9{{%
\run #2
\rise #3
\ifnum\rise=0
  \puthmorphism(#1)[#4`#5`#6]{#7}{#8}#9%
\else\ifnum\run=0
  \putvmorphism(#1)[#4`#5`#6]{#7}{#8}#9%
\else
\setpos(#1)%
\arrowlength #7
\arrowtype #8
\ifnum\run=0
\else\ifnum\rise=0
\else
\ifnum\run>0
    \coefa=1
\else
   \coefa=-1
\fi
\ifnum\arrowtype>0
   \coefb=0
   \coefc=-1
\else
   \coefb=\coefa
   \coefc=1
   \arrowtype=-\arrowtype
\fi
\width=2
\multiply \width by\run
\divide \width by\rise
\ifnum \width<0  \width=-\width\fi
\advance\width by60
\if l#9 \width=-\width\fi
\putbox(\xpos,\ypos){#4}
{\multiply \coefa by\arrowlength
\advance\xpos by\coefa
\multiply \coefa by\rise
\divide \coefa by\run
\advance \ypos by\coefa
\putbox(\xpos,\ypos){#5} }%
{\multiply \coefa by\arrowlength
\divide \coefa by2
\advance \xpos by\coefa
\advance \xpos by\width
\multiply \coefa by\rise
\divide \coefa by\run
\advance \ypos by\coefa
\if l#9%
   \putrbox(\xpos,\ypos){#6}%
\else\if r#9%
   \putlbox(\xpos,\ypos){#6}%
\fi\fi }%
{\multiply \rise by-\coefc
\multiply \run by-\coefc
\multiply \coefb by\arrowlength
\advance \xpos by\coefb
\multiply \coefb by\rise
\divide \coefb by\run
\advance \ypos by\coefb
\multiply \coefc by70
\advance \ypos by\coefc
\multiply \coefc by\run
\divide \coefc by\rise
\advance \xpos by\coefc
\multiply \coefa by140
\multiply \coefa by\run
\divide \coefa by\rise
\advance \arrowlength by\coefa
\ifcase\arrowtype
\or \put(\xpos,\ypos){\vector(\run,\rise){\arrowlength}}%
\or \put(\xpos,\ypos){\mvector(\run,\rise){\arrowlength}}%
\or \put(\xpos,\ypos){\evector(\run,\rise){\arrowlength}}%
\fi}\fi\fi\fi\fi}}

\newcount\numbdashes \newcount\lengthdash \newcount\increment
 
\def\howmanydashes{
\numbdashes=\arrowlength \lengthdash=40
\divide\numbdashes by \lengthdash
\lengthdash=\arrowlength
\divide\lengthdash by \numbdashes
\increment=\lengthdash
\multiply\lengthdash by 3
\divide\lengthdash by 5
}
 
\def\putdashvector(#1)(#2,#3)#4#5{%
\ifnum#3=0 \putdashhvector(#1){#4}#5
\else
\ifnum#2=0
\putdashvvector(#1){#4}#5\fi\fi}
 
\def\putdashhvector(#1,#2)#3#4{{%
\arrowlength=#3 \howmanydashes
\multiput(#1,#2)(\increment,0){\numbdashes}%
{\vrule height .4pt width \lengthdash\unitlength}
\arrowtype=#4 \xpos=#1
\ifnum\arrowtype<0 \advance\arrowtype by 7 \fi
\ifcase\arrowtype
\or \advance\xpos by 10
    \put(\xpos,#2){\vector(-1,0){\lengthdash}}
    \advance\xpos by 40
    \put(\xpos,#2){\vector(-1,0){\lengthdash}}
\or \advance \xpos by 10
    \put(\xpos,#2){\vector(-1,0){\lengthdash}}
    \advance\xpos by  \arrowlength
    \advance\xpos by  -50
    \put(\xpos,#2){\vector(-1,0){\lengthdash}}
\or \advance\xpos by 10
    \put(\xpos,#2){\vector(-1,0){\lengthdash}}
\or \advance\xpos by \arrowlength
    \advance\xpos by -\lengthdash
    \put(\xpos,#2){\vector(1,0){\lengthdash}}
\or {\advance\xpos by 10
    \put(\xpos,#2){\vector(1,0){\lengthdash}}}
    \advance\xpos by \arrowlength
    \advance\xpos by -\lengthdash
    \put(\xpos,#2){\vector(1,0){\lengthdash}}
\or \advance\xpos by \arrowlength
    \advance\xpos by -\lengthdash
    \put(\xpos,#2){\vector(1,0){\lengthdash}}
    \advance\xpos by -40
    \put(\xpos,#2){\vector(1,0){\lengthdash}}
   \fi
}}
 
\def\putdashvvector(#1,#2)#3#4{{%
\arrowlength=#3 \howmanydashes
\ypos=#2 \advance\ypos by -\arrowlength
\multiput(#1,#2)(0,\increment){\numbdashes}%
    {\vrule width .4pt height \lengthdash\unitlength}
\arrowtype=#4 \ypos=#2
\ifnum\arrowtype<0 \advance\arrowtype by 7 \fi
\ifcase\arrowtype
\or \advance\ypos by \arrowlength \advance\ypos by -40
    \put(#1,\ypos){\vector(0,1){\lengthdash}}
    \advance\ypos by -40
    \put(#1,\ypos){\vector(0,1){\lengthdash}}
\or \advance\ypos by 10
    \put(#1,\ypos){\vector(0,1){\lengthdash}}
    \advance\ypos by \arrowlength \advance\ypos by -40
    \put(#1,\ypos){\vector(0,1){\lengthdash}}
\or \advance\ypos by \arrowlength \advance\ypos by -40
    \put(#1,\ypos){\vector(0,1){\lengthdash}}
\or \advance\ypos by 10
    \put(#1,\ypos){\vector(0,-1){\lengthdash}}
\or \advance\ypos by 10
    \put(#1,\ypos){\vector(0,-1){\lengthdash}}
    \advance\ypos by \arrowlength \advance\ypos by -40
    \put(#1,\ypos){\vector(0,-1){\lengthdash}}
\or \advance\ypos by 10
    \put(#1,\ypos){\vector(0,-1){\lengthdash}}
    \advance\ypos by 40
    \put(#1,\ypos){\vector(0,-1){\lengthdash}}
\fi
}}
 
\def\puthmorphism(#1,#2)[#3`#4`#5]#6#7#8{{%
\xpos #1
\ypos #2
\width #6
\arrowlength #6
\arrowtype=#7
\putbox(\xpos,\ypos){#3\vphantom{#4}}%
{\advance \xpos by\arrowlength
\putbox(\xpos,\ypos){\vphantom{#3}#4}}%
\horsize{\tempcounta}{#3}%
\horsize{\tempcountb}{#4}%
\divide \tempcounta by2
\divide \tempcountb by2
\advance \tempcounta by30
\advance \tempcountb by30
\advance \xpos by\tempcounta
\advance \arrowlength by-\tempcounta
\advance \arrowlength by-\tempcountb
\putvector(\xpos,\ypos)(1,0)\arrowlength\arrowtype
\divide \arrowlength by2
\advance \xpos by\arrowlength
\vertsize{\tempcounta}{#5}%
\divide\tempcounta by2
\advance \tempcounta by20
\if a#8 %
   \advance \ypos by\tempcounta
   \putbox(\xpos,\ypos){#5}%
\else
   \advance \ypos by-\tempcounta
   \putbox(\xpos,\ypos){#5}%
\fi}}
 
\def\putvmorphism(#1,#2)[#3`#4`#5]#6#7#8{{%
\xpos #1
\ypos #2
\arrowlength #6
\arrowtype #7
\settowidth{\xlen}{$#5$}%
\putbox(\xpos,\ypos){#3}%
{\advance \ypos by-\arrowlength
\putbox(\xpos,\ypos){#4}}%
{\advance\arrowlength by-140
\advance \ypos by-70
\ifdim\xlen>0pt
   \if m#8%
      \putsplitvector(\xpos,\ypos)\arrowlength\arrowtype
   \else
   \putvector(\xpos,\ypos)(0,-1)\arrowlength\arrowtype
   \fi
\else
   \putvector(\xpos,\ypos)(0,-1)\arrowlength\arrowtype
\fi}%
\ifdim\xlen>0pt
   \divide \arrowlength by2
   \advance\ypos by-\arrowlength
   \if l#8%
      \advance \xpos by-40
      \putrbox(\xpos,\ypos){#5}%
   \else\if r#8%
      \advance \xpos by40
      \putlbox(\xpos,\ypos){#5}%
   \else
      \putbox(\xpos,\ypos){#5}%
   \fi\fi
\fi
}}
 
\def\putsquarep<#1>(#2)[#3;#4`#5`#6`#7]{{%
\setsqparms[#1]%
\setpos(#2)%
\settokens`#3`%
\puthmorphism(\xpos,\ypos)[\tokenc`\tokend`{#7}]{\width}{\arrowtyped}b%
\advance\ypos by \height
\puthmorphism(\xpos,\ypos)[\tokena`\tokenb`{#4}]{\width}{\arrowtypea}a%
\putvmorphism(\xpos,\ypos)[``{#5}]{\height}{\arrowtypeb}l%
\advance\xpos by \width
\putvmorphism(\xpos,\ypos)[``{#6}]{\height}{\arrowtypec}r%
}}
 
\def\putsquare{\@ifnextchar <{\putsquarep}{\putsquarep%
   <\arrowtypea`\arrowtypeb`\arrowtypec`\arrowtyped;\width`\height>}}
\def\square{\@ifnextchar< {\squarep}{\squarep
   <\arrowtypea`\arrowtypeb`\arrowtypec`\arrowtyped;\width`\height>}}
\def\squarep<#1>[#2`#3`#4`#5;#6`#7`#8`#9]{{
\setsqparms[#1]
\diagram
\putsquarep<\arrowtypea`\arrowtypeb`\arrowtypec`
\arrowtyped;\width`\height>
(0,0)[#2`#3`#4`{#5};#6`#7`#8`{#9}]
\enddiagram
}}                                                 
\def\putptrianglep<#1>(#2,#3)[#4`#5`#6;#7`#8`#9]{{%
\settriparms[#1]%
\xpos=#2 \ypos=#3
\advance\ypos by \height
\puthmorphism(\xpos,\ypos)[#4`#5`{#7}]{\height}{\arrowtypea}a%
\putvmorphism(\xpos,\ypos)[`#6`{#8}]{\height}{\arrowtypeb}l%
\advance\xpos by\height
\putmorphism(\xpos,\ypos)(-1,-1)[``{#9}]{\height}{\arrowtypec}r%
}}
 
\def\putptriangle{\@ifnextchar <{\putptrianglep}{\putptrianglep
   <\arrowtypea`\arrowtypeb`\arrowtypec;\height>}}
\def\ptriangle{\@ifnextchar <{\ptrianglep}{\ptrianglep
   <\arrowtypea`\arrowtypeb`\arrowtypec;\height>}}
\def\ptrianglep<#1>[#2`#3`#4;#5`#6`#7]{{
\settriparms[#1]
\diagram
\putptrianglep<\arrowtypea`\arrowtypeb`
\arrowtypec;\height>
(0,0)[#2`#3`#4;#5`#6`{#7}]
\enddiagram
}}                                            
 
\def\putqtrianglep<#1>(#2,#3)[#4`#5`#6;#7`#8`#9]{{%
\settriparms[#1]%
\xpos=#2 \ypos=#3
\advance\ypos by\height
\puthmorphism(\xpos,\ypos)[#4`#5`{#7}]{\height}{\arrowtypea}a%
\putmorphism(\xpos,\ypos)(1,-1)[``{#8}]{\height}{\arrowtypeb}l%
\advance\xpos by\height
\putvmorphism(\xpos,\ypos)[`#6`{#9}]{\height}{\arrowtypec}r%
}}
 
\def\putqtriangle{\@ifnextchar <{\putqtrianglep}{\putqtrianglep
   <\arrowtypea`\arrowtypeb`\arrowtypec;\height>}}
\def\qtriangle{\@ifnextchar <{\qtrianglep}{\qtrianglep
   <\arrowtypea`\arrowtypeb`\arrowtypec;\height>}}
\def\qtrianglep<#1>[#2`#3`#4;#5`#6`#7]{{
\settriparms[#1]
\width=\height                                
\diagram
\putqtrianglep<\arrowtypea`\arrowtypeb`
\arrowtypec;\height>
(0,0)[#2`#3`#4;#5`#6`{#7}]
\enddiagram
}}
 
\def\putdtrianglep<#1>(#2,#3)[#4`#5`#6;#7`#8`#9]{{%
\settriparms[#1]%
\xpos=#2 \ypos=#3
\puthmorphism(\xpos,\ypos)[#5`#6`{#9}]{\height}{\arrowtypec}b%
\advance\xpos by \height \advance\ypos by\height
\putmorphism(\xpos,\ypos)(-1,-1)[``{#7}]{\height}{\arrowtypea}l%
\putvmorphism(\xpos,\ypos)[#4``{#8}]{\height}{\arrowtypeb}r%
}}
 
\def\putdtriangle{\@ifnextchar <{\putdtrianglep}{\putdtrianglep
   <\arrowtypea`\arrowtypeb`\arrowtypec;\height>}}
\def\dtriangle{\@ifnextchar <{\dtrianglep}{\dtrianglep
   <\arrowtypea`\arrowtypeb`\arrowtypec;\height>}}
\def\dtrianglep<#1>[#2`#3`#4;#5`#6`#7]{{
\settriparms[#1]
\width=\height                                
\diagram
\putdtrianglep<\arrowtypea`\arrowtypeb`
\arrowtypec;\height>
(0,0)[#2`#3`#4;#5`#6`{#7}]
\enddiagram
}}
 
\def\putbtrianglep<#1>(#2,#3)[#4`#5`#6;#7`#8`#9]{{%
\settriparms[#1]%
\xpos=#2 \ypos=#3
\puthmorphism(\xpos,\ypos)[#5`#6`{#9}]{\height}{\arrowtypec}b%
\advance\ypos by\height
\putmorphism(\xpos,\ypos)(1,-1)[``{#8}]{\height}{\arrowtypeb}r%
\putvmorphism(\xpos,\ypos)[#4``{#7}]{\height}{\arrowtypea}l%
}}
 
\def\putbtriangle{\@ifnextchar <{\putbtrianglep}{\putbtrianglep
   <\arrowtypea`\arrowtypeb`\arrowtypec;\height>}}
\def\btriangle{\@ifnextchar <{\btrianglep}{\btrianglep
   <\arrowtypea`\arrowtypeb`\arrowtypec;\height>}}
\def\btrianglep<#1>[#2`#3`#4;#5`#6`#7]{{
\settriparms[#1]
\width=\height                               
\diagram
\putbtrianglep<\arrowtypea`\arrowtypeb`
\arrowtypec;\height>
(0,0)[#2`#3`#4;#5`#6`{#7}]
\enddiagram
}}
 
\def\putAtrianglep<#1>(#2,#3)[#4`#5`#6;#7`#8`#9]{{%
\settriparms[#1]%
\xpos=#2 \ypos=#3
{\multiply \height by2
\puthmorphism(\xpos,\ypos)[#5`#6`{#9}]{\height}{\arrowtypec}b}%
\advance\xpos by\height \advance\ypos by\height
\putmorphism(\xpos,\ypos)(-1,-1)[#4``{#7}]{\height}{\arrowtypea}l%
\putmorphism(\xpos,\ypos)(1,-1)[``{#8}]{\height}{\arrowtypeb}r%
}}
 
\def\putAtriangle{\@ifnextchar <{\putAtrianglep}{\putAtrianglep
   <\arrowtypea`\arrowtypeb`\arrowtypec;\height>}}
\def\Atriangle{\@ifnextchar <{\Atrianglep}{\Atrianglep
   <\arrowtypea`\arrowtypeb`\arrowtypec;\height>}}
\def\Atrianglep<#1>[#2`#3`#4;#5`#6`#7]{{
\settriparms[#1]
\width=\height                                     
\diagram
\putAtrianglep<\arrowtypea`\arrowtypeb`
\arrowtypec;\height>
(0,0)[#2`#3`#4;#5`#6`{#7}]
\enddiagram
}}
 
\def\putAtrianglepairp<#1>(#2)[#3;#4`#5`#6`#7`#8]{{%
\settripairparms[#1]%
\setpos(#2)%
\settokens`#3`%
\puthmorphism(\xpos,\ypos)[\tokenb`\tokenc`{#7}]{\height}{\arrowtyped}b%
\advance\xpos by\height
\puthmorphism(\xpos,\ypos)[\phantom{\tokenc}`\tokend`{#8}]%
{\height}{\arrowtypee}b%
\advance\ypos by\height
\putmorphism(\xpos,\ypos)(-1,-1)[\tokena``{#4}]{\height}{\arrowtypea}l%
\putvmorphism(\xpos,\ypos)[``{#5}]{\height}{\arrowtypeb}m%
\putmorphism(\xpos,\ypos)(1,-1)[``{#6}]{\height}{\arrowtypec}r%
}}
 
\def\putAtrianglepair{\@ifnextchar <{\putAtrianglepairp}{\putAtrianglepairp%
   <\arrowtypea`\arrowtypeb`\arrowtypec`\arrowtyped`\arrowtypee;\height>}}
\def\Atrianglepair{\@ifnextchar <{\Atrianglepairp}{\Atrianglepairp%
   <\arrowtypea`\arrowtypeb`\arrowtypec`\arrowtyped`\arrowtypee;\height>}}
 
\def\Atrianglepairp<#1>[#2;#3`#4`#5`#6`#7]{{
\settripairparms[#1]
\settokens`#2`
\width=\height                                
\diagram
\putAtrianglepairp                            
<\arrowtypea`\arrowtypeb`\arrowtypec`
\arrowtyped`\arrowtypee;\height>
(0,0)[{#2};#3`#4`#5`#6`{#7}]
\enddiagram
}}
 
\def\putVtrianglep<#1>(#2,#3)[#4`#5`#6;#7`#8`#9]{{%
\settriparms[#1]%
\xpos=#2 \ypos=#3
\advance\ypos by\height
{\multiply\height by2
\puthmorphism(\xpos,\ypos)[#4`#5`{#7}]{\height}{\arrowtypea}a}%
\putmorphism(\xpos,\ypos)(1,-1)[`#6`{#8}]{\height}{\arrowtypeb}l%
\advance\xpos by\height
\advance\xpos by\height
\putmorphism(\xpos,\ypos)(-1,-1)[``{#9}]{\height}{\arrowtypec}r%
}}
 
\def\putVtriangle{\@ifnextchar <{\putVtrianglep}{\putVtrianglep
   <\arrowtypea`\arrowtypeb`\arrowtypec;\height>}}
\def\Vtriangle{\@ifnextchar <{\Vtrianglep}{\Vtrianglep
   <\arrowtypea`\arrowtypeb`\arrowtypec;\height>}}
\def\Vtrianglep<#1>[#2`#3`#4;#5`#6`#7]{{
\settriparms[#1]
\width=\height                                 
\diagram
\putVtrianglep<\arrowtypea`\arrowtypeb`
\arrowtypec;\height>
(0,0)[#2`#3`#4;#5`#6`{#7}]
\enddiagram
}}
 
\def\putVtrianglepairp<#1>(#2)[#3;#4`#5`#6`#7`#8]{{
\settripairparms[#1]%
\setpos(#2)%
\settokens`#3`%
\advance\ypos by\height
\putmorphism(\xpos,\ypos)(1,-1)[`\tokend`{#6}]{\height}{\arrowtypec}l%
\puthmorphism(\xpos,\ypos)[\tokena`\tokenb`{#4}]{\height}{\arrowtypea}a%
\advance\xpos by\height
\puthmorphism(\xpos,\ypos)[\phantom{\tokenb}`\tokenc`{#5}]%
{\height}{\arrowtypeb}a%
\putvmorphism(\xpos,\ypos)[``{#7}]{\height}{\arrowtyped}m%
\advance\xpos by\height
\putmorphism(\xpos,\ypos)(-1,-1)[``{#8}]{\height}{\arrowtypee}r%
}}
 
\def\putVtrianglepair{\@ifnextchar <{\putVtrianglepairp}{\putVtrianglepairp%
    <\arrowtypea`\arrowtypeb`\arrowtypec`\arrowtyped`\arrowtypee;\height>}}
\def\Vtrianglepair{\@ifnextchar <{\Vtrianglepairp}{\Vtrianglepairp%
    <\arrowtypea`\arrowtypeb`\arrowtypec`\arrowtyped`\arrowtypee;\height>}}
\def\Vtrianglepairp<#1>[#2;#3`#4`#5`#6`#7]{{
\settripairparms[#1]
\settokens`#2`
\diagram
\putVtrianglepairp                             
<\arrowtypea`\arrowtypeb`\arrowtypec`
\arrowtyped`\arrowtypee;\height>
(0,0)[{#2};#3`#4`#5`#6`{#7}]
\enddiagram
}}

\def\putCtrianglep<#1>(#2,#3)[#4`#5`#6;#7`#8`#9]{{%
\settriparms[#1]%
\xpos=#2 \ypos=#3
\advance\ypos by\height
\putmorphism(\xpos,\ypos)(1,-1)[``{#9}]{\height}{\arrowtypec}l%
\advance\xpos by\height
\advance\ypos by\height
\putmorphism(\xpos,\ypos)(-1,-1)[#4`#5`{#7}]{\height}{\arrowtypea}l%
{\multiply\height by 2
\putvmorphism(\xpos,\ypos)[`#6`{#8}]{\height}{\arrowtypeb}r}%
}}
 
\def\putCtriangle{\@ifnextchar <{\putCtrianglep}{\putCtrianglep
    <\arrowtypea`\arrowtypeb`\arrowtypec;\height>}}
\def\Ctriangle{\@ifnextchar <{\Ctrianglep}{\Ctrianglep
    <\arrowtypea`\arrowtypeb`\arrowtypec;\height>}}
\def\Ctrianglep<#1>[#2`#3`#4;#5`#6`#7]{{
\settriparms[#1]
\width=\height                               
\diagram
\putCtrianglep<\arrowtypea`\arrowtypeb`
\arrowtypec;\height>
(0,0)[#2`#3`#4;#5`#6`{#7}]
\enddiagram
}}                                           
\def\putDtrianglep<#1>(#2,#3)[#4`#5`#6;#7`#8`#9]{{%
\settriparms[#1]%
\xpos=#2 \ypos=#3
\advance\xpos by\height \advance\ypos by\height
\putmorphism(\xpos,\ypos)(-1,-1)[``{#9}]{\height}{\arrowtypec}r%
\advance\xpos by-\height \advance\ypos by\height
\putmorphism(\xpos,\ypos)(1,-1)[`#5`{#8}]{\height}{\arrowtypeb}r%
{\multiply\height by 2
\putvmorphism(\xpos,\ypos)[#4`#6`{#7}]{\height}{\arrowtypea}l}%
}}
 
\def\putDtriangle{\@ifnextchar <{\putDtrianglep}{\putDtrianglep
    <\arrowtypea`\arrowtypeb`\arrowtypec;\height>}}
\def\Dtriangle{\@ifnextchar <{\Dtrianglep}{\Dtrianglep
   <\arrowtypea`\arrowtypeb`\arrowtypec;\height>}}
\def\Dtrianglep<#1>[#2`#3`#4;#5`#6`#7]{{
\settriparms[#1]
\width=\height                              
\diagram
\putDtrianglep<\arrowtypea`\arrowtypeb`
\arrowtypec;\height>
(0,0)[#2`#3`#4;#5`#6`{#7}]
\enddiagram
}}                                          
\def\setrecparms[#1`#2]{\width=#1 \height=#2}%
 
\def\recursep<#1`#2>[#3;#4`#5`#6`#7`#8]{{\m@th
\width=#1 \height=#2
\settokens`#3`
\settowidth{\tempdimen}{$\tokena$}
\ifdim\tempdimen=0pt
  \savebox{\tempboxa}{\hbox{$\tokenb$}}%
  \savebox{\tempboxb}{\hbox{$\tokend$}}%
  \savebox{\tempboxc}{\hbox{$#6$}}%
\else
  \savebox{\tempboxa}{\hbox{$\hbox{$\tokena$}\times\hbox{$\tokenb$}$}}%
  \savebox{\tempboxb}{\hbox{$\hbox{$\tokena$}\times\hbox{$\tokend$}$}}%
  \savebox{\tempboxc}{\hbox{$\hbox{$\tokena$}\times\hbox{$#6$}$}}%
\fi
\ypos=\height
\divide\ypos by 2
\xpos=\ypos
\advance\xpos by \width
\bfig
\putCtrianglep<-1`1`1;\ypos>(0,0)[`\tokenc`;#5`#6`{#7}]%
\puthmorphism(\ypos,0)[\tokend`\usebox{\tempboxb}`{#8}]{\width}{-1}b%
\puthmorphism(\ypos,\height)[\tokenb`\usebox{\tempboxa}`{#4}]{\width}{-1}a%
\advance\ypos by \width
\putvmorphism(\ypos,\height)[``\usebox{\tempboxc}]{\height}1r%
\efig
}}
 
\def\recurse{\@ifnextchar <{\recursep}{\recursep<\width`\height>}}
 
\def\puttwohmorphisms(#1,#2)[#3`#4;#5`#6]#7#8#9{{%
\puthmorphism(#1,#2)[#3`#4`]{#7}0a
\ypos=#2
\advance\ypos by 20
\puthmorphism(#1,\ypos)[\phantom{#3}`\phantom{#4}`#5]{#7}{#8}a
\advance\ypos by -40
\puthmorphism(#1,\ypos)[\phantom{#3}`\phantom{#4}`#6]{#7}{#9}b
}}
 
\def\puttwovmorphisms(#1,#2)[#3`#4;#5`#6]#7#8#9{{%
\putvmorphism(#1,#2)[#3`#4`]{#7}0a
\xpos=#1
\advance\xpos by -20
\putvmorphism(\xpos,#2)[\phantom{#3}`\phantom{#4}`#5]{#7}{#8}l
\advance\xpos by 40
\putvmorphism(\xpos,#2)[\phantom{#3}`\phantom{#4}`#6]{#7}{#9}r
}}
 
\def\puthcoequalizer(#1)[#2`#3`#4;#5`#6`#7]#8#9{{%
\setpos(#1)%
\puttwohmorphisms(\xpos,\ypos)[#2`#3;#5`#6]{#8}11%
\advance\xpos by #8
\puthmorphism(\xpos,\ypos)[\phantom{#3}`#4`#7]{#8}1{#9}
}}
 
\def\putvcoequalizer(#1)[#2`#3`#4;#5`#6`#7]#8#9{{%
\setpos(#1)%
\puttwovmorphisms(\xpos,\ypos)[#2`#3;#5`#6]{#8}11%
\advance\ypos by -#8
\putvmorphism(\xpos,\ypos)[\phantom{#3}`#4`#7]{#8}1{#9}
}}
 
\def\putthreehmorphisms(#1)[#2`#3;#4`#5`#6]#7(#8)#9{{%
\setpos(#1) \settypes(#8)
\if a#9 %
     \vertsize{\tempcounta}{#5}%
     \vertsize{\tempcountb}{#6}%
     \ifnum \tempcounta<\tempcountb \tempcounta=\tempcountb \fi
\else
     \vertsize{\tempcounta}{#4}%
     \vertsize{\tempcountb}{#5}%
     \ifnum \tempcounta<\tempcountb \tempcounta=\tempcountb \fi
\fi
\advance \tempcounta by 60
\puthmorphism(\xpos,\ypos)[#2`#3`#5]{#7}{\arrowtypeb}{#9}
\advance\ypos by \tempcounta
\puthmorphism(\xpos,\ypos)[\phantom{#2}`\phantom{#3}`#4]{#7}{\arrowtypea}{#9}
\advance\ypos by -\tempcounta \advance\ypos by -\tempcounta
\puthmorphism(\xpos,\ypos)[\phantom{#2}`\phantom{#3}`#6]{#7}{\arrowtypec}{#9}
}}
 
\def\setarrowtoks[#1`#2`#3`#4`#5`#6]{%
\def\toka{#1}
\def\tokb{#2}
\def\tokc{#3}
\def\tokd{#4}
\def\toke{#5}
\def\tokf{#6}
}
\def\hex{\@ifnextchar <{\hexp}{\hexp<1000`400>}}
\def\hexp<#1`#2>[#3`#4`#5`#6`#7`#8;#9]{%
\setarrowtoks[#9]
\yext=#2 \advance \yext by #2
\xext=#1 \advance\xext by \yext
\bfig
\putCtriangle<-1`0`1;#2>(0,0)[`#5`;\tokb``\tokd]
\xext=#1 \yext=#2 \advance \yext by #2
\putsquare<1`0`0`1;\xext`\yext>(#2,0)[#3`#4`#7`#8;\toka```\tokf]
\advance \xext by #2
\putDtriangle<0`1`-1;#2>(\xext,0)[`#6`;`\tokc`\toke]
\efig
}
\makeatother

\input{tcilatex}
\begin{document}
\title[Some Exotic Characteristic Homomorphism for Lie Algebroids]{Some
Exotic Characteristic Homomorphism for\\
Lie Algebroids}
\author{Bogdan Balcerzak}
\address{Institute of Mathematics, Technical University of \L \'{o}d\'{z},
ul. W\'{o}lcza\'{n}ska 215, 90-924 \L \'{o}d\'{z}, Poland}
\email[B. Balcerzak]{ bogdan.balcerzak@p.lodz.pl}
\author{Jan Kubarski}
\email[J. Kubarski]{ jan.kubarski@p.lodz.pl}

\begin{abstract}
The authors define some secondary characteristic homomorphism for the triple 
$\left( A,B,\nabla \right) $, in which $B\subset A$ is a pair of regular Lie
algebroids over the same foliated manifold and $\nabla :L\rightarrow A$ is a
homomorphism of Lie algebroids (i.e. a flat $L$-connection in $A$) where $L$
is an arbitrary (not necessarily regular) Lie algebroid and show that
characteristic classes from its image generalize known exotic characteristic
classes for flat regular Lie algebroids (Kubarski) and flat principal fibre
bundles with a reduction (Kamber, Tondeur). The generalization includes also
the one given by Crainic for representations of Lie algebroids on vector
bundles. For a pair of regular Lie algebroids $B\subset A$ and for the
special case of the flat connection $\limfunc{id}_{A}:A\rightarrow A$ we
obtain a characteristic homomorphism which is universal in the sense that it
is a factor of any other one for an arbitrary flat $L$-connection $\nabla
:L\rightarrow A$.
\end{abstract}

\maketitle

\section{Introduction}

From the very beginning the characteristic classes (primary and secondary)
are global invariants of geometric structures on manifolds (more generally
on principal fibre bundles on manifolds) determined mainly by connections,
reductions of structure Lie groups, and so on, and having some important
topological properties like homotopy independence, functoriality or
rigidity. N.~Teleman \cite{Teleman} showed in 1972 that the Chern-Weil
homomorphism of any principal fibre bundle with a \textbf{connected}
structural Lie group is an invariant of its infinitesimal object, i.e. of
the Lie algebroid of this bundle. J.~Kubarski \cite{K1} showed in 1991 that
the condition of the connectedness of the structural Lie group is redundant,
which means that primary characteristic homomorphisms (i.e. the Chern-Weil
homomorphism) is really the \textquotedblleft algebraic\textquotedblright\
notion belonging to the category of Lie algebroids. The crucial role was
played by some generalization of the standard concepts of the representation
of Lie groups (and Lie algebras) on vector spaces to the concept of the
representation of principal fibre bundles and of Lie algebroids on vector
bundles and comparing the spaces of suitable invariant cross-sections. It
turns out (J.~Kubarski \cite{K5}, \cite{K6})) that the same holds for the
secondary (exotic) characteristic classes, in particular, for the
characteristic classes of flat bundles. In \cite{K5} there was constructed a
characteristic homomorphism for flat regular Lie algebroids equipped with
some \textquotedblleft reduction\textquotedblright , i.e. with some Lie
subalgebroid, generalizing this homomorphism for foliated principal fibre
bundles given by F.~Kamber and Ph.~Tondeur \cite{K-T1,K-T2,K-T3}. Next, a
different approaches to secondary classes by M.~Crainic and R.~L.~Fernandes 
\cite[2003]{Cr-1}, \cite{Cr-2}, \cite{F}, \cite[2005]{Cr-F}, appeared in the
geometry of Lie algebroids (inspired, for example, by irregular Lie
algebroids important in the Poisson geometry), for example, secondary
characteristic classes for representations \cite[2003]{Cr-1}, characteristic
classes up to homotopy \cite{Cr-2} and intrinsic secondary characteristic
classes \cite{F}, \cite[2005]{Cr-F}. The last were lastly generalized by
I.~Vaisman \cite{V}.

The main goal of the paper is to build an exotic characteristic homomorphism
in the category of Lie algebroids, which generalizes simultaneously the one
given by Kubarski \cite{K5} and the one given by Crainic \cite{Cr-1} and
describes in the Lie algebroids language the classical case for foliated
principal bundles. It is a characteristic homomorphism $\Delta _{\left(
A,B,\nabla \right) \#}$ for the triple $\left( A,B,\nabla \right) $, where $%
B\subset {A}$ is a pair of regular Lie algebroids, both on the same regular
foliated manifold $\left( M,F\right) $, and $\nabla :L\rightarrow A$ is a
flat $L$-connection in $A$ (i.e. $\nabla $ is a homomorphism of Lie
algebroids), where $L$ is an arbitrary (not necessarily regular) Lie
algebroid on $M$. For a Lie algebroid of a principal fibre bundle, its
reduction, and a usual flat connection, we obtain a homomorphism equivalent
to the one by Kamber and Tondeur, mentioned above. Putting $L=A$ and the
\textquotedblleft trivial\textquotedblright\ flat connection $\nabla =%
\limfunc{id}_{A}:A\rightarrow A$, we obtain a quite new characteristic
homomorphism previously unknown, even in the context of the principal fibre
bundles. In fact, this homomorphism is obtained for a pair of Lie algebroids 
$\left( A,B\right) $, $B\subset A$ and can be denoted by $\Delta _{\left(
A,B\right) \#}$. It is in some sense a universal homomorphism. Namely, for
any flat $L$-connection $\nabla :L\rightarrow A$ we have%
\begin{equation}
\Delta _{\left( A,B,\nabla \right) \#}=\nabla ^{\#}\circ \Delta _{\left(
A,B\right) \#}
\end{equation}%
(we will say that $\Delta _{\left( A,B,\nabla \right) \#}$ is \emph{%
factorized} by the universal characteristic homomorphism $\Delta _{\left(
A,B\right) \#}$). The classes form the image of $\Delta _{\left( A,B\right)
\#}$ (which belongs to the cohomology algebra $\mathsf{H}^{\bullet }\left(
A\right) $ ) are called \emph{universal} for the pair $B\subset A.$

In the context of the comparison with the Crainic classes (which concern [in
our setting] the triple $\left( \limfunc{A}\left( \mathfrak{f}\right) ,%
\limfunc{A}\left( \mathfrak{f,}\left\{ h\right\} \right) ,\nabla \right) $
where $\limfunc{A}\left( \mathfrak{f}\right) $, $\limfunc{A}\left( \mathfrak{%
f,}\left\{ h\right\} \right) $ are Lie algebroids of a vector bundle $%
\mathfrak{f}$, its Riemann reduction $\left( \mathfrak{f,}\left\{ h\right\}
\right) $, and $\nabla :L\rightarrow \limfunc{A}\left( \mathfrak{f}\right) $
is an arbitrary flat $L$-connection in $\mathfrak{f}$,{\Huge \ }i.e. a
representation of $L$ on $\mathfrak{f}$), we present one -- based on the
Pfaffian -- characteristic class for an even dimensional, oriented vector
bundle not considered by Crainic. An example with such a nontrivial
universal characteristic class is presented in Section \ref{Example_proof}.

\section{The Secondary (Exotic) Characteristic Homomorphism for FS-Lie
Algebroids and the Universal Characteristic Homomorphism\label%
{Unifying_homomorphism}}

In this section we first define a characteristic homomorphism for the triple 
$\left( A,B,\nabla \right) $, where $B\subset A$ is a pair of regular Lie
algebroids, both on the same regular foliated manifold $\left( M,F\right) $,
and $\nabla :L\rightarrow A$ is a flat $L$-connection in $A$, where $L$ is
an arbitrary Lie algebroid on a manifold $M$.\ Next we compare this
homomorphism with characteristic homomorphisms for regular Lie algebroids
and usual flat connections -- given by Kubarski \cite{K5}, and for principal
fibre bundles -- given by Kamber and Tondeur \cite{K-T3}. The comparison
with the Crainic approach of secondary characteristic classes for
representations will be given in the next section.

\subsection{A Few Words about Lie Algebroids}

The notion of Lie algebroid (Pradines, 1967) had appeared as an
infinitesimal object of Lie groupoids, principal bundles, vector bundles,
TC-foliations, Poisson and Jacobi manifolds, etc. (for the historical
approach see \cite{M2}, \cite{M1}). A \textit{Lie algebroid} over a smooth
manifold $M$ is a triple $\left( {L},[\![\cdot ,\cdot ]\!],\#_{L}\right) $
where $L$ is a vector bundle over $M$, $\left( \Gamma \left( {L}\right)
,[\![\cdot ,\cdot ]\!]\right) $ is an $\mathbb{R}$-Lie algebra, $%
\#_{L}:L\rightarrow $ $TM$ is a linear homomorphism of vector bundles such
that $[\![\xi ,f\cdot \eta ]\!]=f\cdot \lbrack \![\xi ,\eta
]\!]+\#_{L}\left( \xi \right) \left( f\right) \cdot \eta $ for all $f\in
C^{\infty }\left( M\right) $, $\xi ,\eta \in \Gamma \left( {L}\right) $. The
anchor $\#_{L}$ of $\left( {L},[\![\cdot ,\cdot ]\!],\#_{L}\right) $ is
bracket-preserving, see \cite{He} and \cite{B-K-W}. The image $F:=\func{Im}%
\#_{L}\subset TM$ of the anchor $\#_{L}$ is an integrable (non-constant rank
in general) distribution whose leaves form a Stefan foliation $\mathcal{F}$
of $M$ \cite{P2}, \cite{Dufour}. We say also that $A$ is a Lie algebroid
over the foliated manifold $\left( M,F\right) .$ If the anchor $\#_{L}$ is
of constant rank [an epimorphism], then the Lie algebroid $L$ is called 
\textit{regular} [\textit{transitive}]. By a homomorphism of Lie algebroids $%
T:L\rightarrow A$ on a given manifold we mean a linear homomorphism of the
underlying vector bundles which commutes with the anchors and preserves Lie
bracket. The Lie algebroids of Lie groupoids, principal bundles, vector
bundles and TC-foliations are transitive, but the Lie algebroids of general
differential groupoids, Poisson manifolds, Jacobi manifolds etc. are rather
nontransitive (and, in general, irregular). For details concerning Lie
functors on the category of principal fibre bundles $P\rightsquigarrow 
\limfunc{A}\left( P\right) $ and vector bundles $\mathfrak{f}%
\rightsquigarrow \limfunc{A}\left( \mathfrak{f}\right) $ see, for example, 
\cite{K1}. $\limfunc{A}\left( P\right) =TP/G$ is a Lie algebroid with the
real Lie algebra $\Gamma \left( \limfunc{A}\left( P\right) \right) \cong 
\mathfrak{X}^{r}\left( P\right) $ and the anchor $\#_{\limfunc{A}\left(
P\right) }:\limfunc{A}\left( P\right) \rightarrow $ $TM$ determined by the
projection $\pi _{\ast }$ whereas $\limfunc{A}\left( \mathfrak{f}\right) $
is the vector bundle whose global cross-sections form a Lie algebra of
covariant derivative operators and the anchor $\#_{\limfunc{A}\left( 
\mathfrak{f}\right) }:\limfunc{A}\left( \mathfrak{f}\right) \rightarrow TM$
is defined by the anchors of these operators. Together with a $V$-vector
bundle $\mathfrak{f}$ (the vector space $V$ is the typical fibre of $%
\mathfrak{f}$) we associate the $GL\left( V\right) $-principal fibre bundle
of frames $\limfunc{L}\left( \mathfrak{f}\right) $ and its Lie algebroid $%
\limfunc{A}\left( \limfunc{L}\left( \mathfrak{f}\right) \right) $ which is
canonically isomorphic to $\limfunc{A}\left( \mathfrak{f}\right) $ \cite{K1}%
. For a regular Lie algebroid $L$ we have the exact Atiyah sequence\ $%
0\longrightarrow {{\pmb{g}\hookrightarrow }}L\overset{\#_{L}}{%
\longrightarrow }F\longrightarrow 0$, (${{\pmb{g}=\ker }}\#_{L}$), the fibre 
${{\pmb{g}}}_{|x}$ of ${{\pmb{g}}}$ at $x$\ is a Lie algebra called the 
\textit{isotropy Lie algebra at }$x$. Over any leaf of the foliation $F$ the
vector bundle ${{\pmb{g}}}$ is a Lie algebra bundle. A splitting of this
sequence $\nabla :F\rightarrow L$ (i.e. $\#_{L}\circ \nabla =\limfunc{id}%
\nolimits_{F}$) is called a connection in $L.$ If $L=\limfunc{A}\left(
P\right) $ (here $F=TM$), then connections in $L$ correspond one-to-one to
usual connections in the principal fibre bundle $P$ (i.e. a horizontal
right-invariant distributions on $P$). We consider more general notion of $L$%
-connection in $A$\ (where $L$\ and $A$\ are arbitrary Lie algebroids on the
same manifold) understanding as a linear homomorphism of vector bundles $%
\nabla :L\rightarrow A$\ commuting with the anchors \cite{B-K-W}. If $A$ is
a regular Lie algebroid with the adjoint LAB $\pmb{g}$, then the curvature
tensor $R^{\nabla }\in \Omega ^{2}\left( L,{{\pmb{g}}}\right) $ is defined
by $R^{\nabla }\left( \xi ,\eta \right) =[\![\nabla \xi ,\nabla \eta
]\!]-\nabla \left( \lbrack \![\xi ,\eta ]\!]\right) ,\ \ \xi ,\eta \in
\Gamma \left( L\right) .$ Clearly, $\nabla $ is a homomorphism of Lie
algebroids if and only if $\nabla $ is flat, i.e. if $R^{\nabla }=0$. Any $L$%
-connection $\nabla :L\rightarrow A\ $determines the standard operator $%
d^{\nabla }:\Omega \left( L;{{\pmb{g}}}\right) \rightarrow \Omega \left( L;{{%
\pmb{g}}}\right) $ in the space of $L$-differential forms with values in ${{%
\pmb{g}}}$ by the formula 
\begin{align}
\left( d^{\nabla }\Omega \right) \left( \xi _{1},\ldots ,\xi _{n}\right) &
=\dsum\nolimits_{i=1}^{n}\left( -1\right) ^{i-1}[\![\nabla _{\xi
_{i}},\Omega \left( \xi _{1},\ldots \hat{\imath}\ldots ,\xi _{n}\right)
]\!]_{A}  \label{dnabla} \\
& +\dsum\nolimits_{i<j}\left( -1\right) ^{i+j}\,\Omega \left( \lbrack \![\xi
_{i},\xi _{j}]\!]_{L},\xi _{1},\ldots \hat{\imath}\ldots \hat{\jmath}\ldots
,\xi _{n}\right) .  \notag
\end{align}%
The equality $d^{\nabla }d^{\nabla }\Omega =R^{\nabla }\wedge \Omega $ holds;%
{\Huge \ }in particular, $d^{\nabla }\ $is a differential operator if $%
\nabla $ is flat. For an arbitrary Lie algebroid $L$ there is a derivative $%
d_{L}$ in the space of real $L$-forms $\Gamma (\dbigwedge L^{\ast })$ giving
the cohomology algebra $\mathsf{H}^{\bullet }\left( L\right) $. For more
details about Lie algebroids we refer the reader, for example, to \cite{M1}, 
\cite{M2}.

\subsection{The Construction of the Secondary Characteristic Classes}

Let us consider the triple $\left( A,B,\nabla \right) $, in which we have: a
regular Lie algebroid $\left( {A},[\![\cdot ,\cdot ]\!],\#_{A}\right) $ on a
foliated manifold $\left( M,F\right) $, its\ regular Lie subalgebroid $%
B\subset A$, also on the same foliated manifold $\left( M,F\right) $, and a
flat $L$-connection $\nabla :L\rightarrow A$ in $A$ for an arbitrary Lie
algebroid $L$. We will call the triple%
\begin{equation*}
\left( A,B,\nabla \right)
\end{equation*}%
an FS-\emph{Lie algebroid}. The characteristic homomorphism

for this triple constructed below measures the independence of these two
geometric structures $B$ and $\nabla $ defined for $A$ (in the sense that it
is zero when $\nabla $ takes values in $B$). In the diagram below $\lambda
:F\rightarrow B$ is an arbitrary auxiliary connection in $B$. Then $j\circ
\lambda :F\rightarrow A$ is a connection in $A$. Let $\breve{\lambda}%
:A\rightarrow {{\pmb{g}}}$ be its connection form. 
\begin{equation*}
\xext=2000\yext=700\begin{picture}(\xext,\yext)(\xoff,\yoff) \putmorphism
(0,620)(1,0)[0`\pmb{g}`]{500}1a \putmorphism(550,620)(1,0)[`A`]{450}1a
\put(750,680){\makebox(0,0){$i$}} \putmorphism(1050,620)(1,0)[``]{450}{-1}a
\putmorphism (1050,50)(1,0)[``{\lambda}]{400}{-1}b
\put(1775,620){\makebox(0,0){$L$}} \put(1780,540){\vector(0,-1){360}}
\put(1450,620){\line(1,0){270}} \putmorphism(0,100)(1,0)[0`\pmb{h}`]{500}1a
\put(444,170){{$\cup$}} \putmorphism(570,100)(1,0)[`B`]{450}1a
\putmorphism(1050,100)(1,0)[`F`\#_{B}]{450}1a
\put(1790,100){\makebox(0,0){$F_{1}$}} \put(540,623){{$\subset$}}
\put(1650,100){\makebox(0,0){$\supset$}} \put(1870,370){\makebox
(0,0){$\#_{L}$}} \putmorphism(504,600)(0,1)[``]{460}{-1}r \put
(550,105){{$\subset$}} \put(951,170){{$\cup$}} \putmorphism
(1010,600)(0,1)[``j]{460}{-1}r \putmorphism(1020,620)(1,-1)[``]{530}1l
\put(1380,370){\makebox(0,0){${\#}_{A}$}} \put(1380,690){\makebox
(0,0){$\nabla$}} \put(935,590){\vector(-1,0){370}} \put(750,530){\makebox
(0,0){{$\breve{\lambda} $}}} \end{picture}
\end{equation*}%
We define the homomorphism 
\begin{equation*}
\omega _{B,\nabla }:L\longrightarrow {{\pmb{g}/\pmb{h}\ \ \ }}\text{by\ \ \ }%
\omega _{B,\nabla }\left( w\right) =[-(\breve{\lambda}\circ \nabla )\left(
w\right) ].
\end{equation*}%
Observe that $\omega _{B,\nabla }$ does not depend on the choice of an
auxiliary connection $\lambda :F\rightarrow A$ and $\omega _{B,\nabla }=0$
if $\nabla $ takes values in $B$. Let us define a homomorphism of algebras 
\begin{equation}
\Delta _{\left( A,B,\nabla \right) }:\Gamma \left(
\bigwedge\nolimits^{k}\left( {{\pmb{g}/\pmb{h}}}\right) ^{\ast }\right)
\longrightarrow \Omega \left( L\right) ,  \label{homodelta}
\end{equation}%
\begin{equation*}
{\large (}\Delta _{\left( A,B,\nabla \right) }\Psi {\large )}_{x}\left(
w_{1}\wedge \ldots \wedge w_{k}\right) =\left\langle \Psi _{x},\omega
_{B,\nabla }\left( w_{1}\right) \wedge \ldots \wedge \omega _{B,\nabla
}\left( w_{k}\right) \right\rangle ,\ \ w_{i}\in L_{|x}.
\end{equation*}%
In the special simple case $L=A\ $and the flat connection $\nabla =\func{id}%
_{A}:A\rightarrow A$ we have particular case of a homomorphism for the pair $%
\left( A,B\right) $:%
\begin{eqnarray*}
\Delta _{\left( A,B\right) } &:&\Gamma \left( \bigwedge\nolimits^{k}\left( {{%
\pmb{g}/\pmb{h}}}\right) ^{\ast }\right) \longrightarrow \Omega \left(
A\right) , \\
{\large (}\Delta _{\left( A,B\right) }\Psi {\large )}_{x}\left( \upsilon
_{1}\wedge \ldots \wedge \upsilon _{k}\right) &=&\langle \Psi _{x},[-\breve{%
\lambda}\left( \upsilon _{1}\right) ]\wedge \ldots \wedge \lbrack -\breve{%
\lambda}\left( \upsilon _{k}\right) ]\rangle ,\ \ \ \ \upsilon _{i}\in
A_{|x}.
\end{eqnarray*}%
We assert that $\Delta _{\left( A,B,\nabla \right) }$ can be written as a
composition $\Delta _{\left( A,B,\nabla \right) }=\nabla ^{\ast }\circ
\Delta _{\left( A,B\right) },$%
\begin{equation*}
\begin{CD}
 \Delta _{\left( A,B,\nabla \right) }:\operatorname{\Gamma}\left( \bigwedge\left( {{\pmb{g}/\pmb{h}}}\right) ^{\ast}\right)
@>{   \Delta _{\left( A,B\right) } }>>
\Omega\left( A\right)    @>{\nabla^{\ast} }>> \Omega\left( L\right),
\end{CD}%
\end{equation*}%
where $\nabla ^{\ast }$ is the pullback of forms. In the algebra $\Gamma
\left( \bigwedge \left( {{\pmb{g}/\pmb{h}}}\right) ^{\ast }\right) $ we
distinguish the subalgebra \label{sec3.1 (A)}$\left( \Gamma \left( \bigwedge
\left( {{\pmb{g}/\pmb{h}}}\right) ^{\ast }\right) \right) ^{\Gamma \left(
B\right) }$ of invariant cross-sections with respect to the representation
of the Lie algebroid $B$ in the vector bundle $\bigwedge \left( {{\pmb{g}/%
\pmb{h}}}\right) ^{\ast }$, associated to the adjoint one $\limfunc{ad}%
\nolimits_{B,{{\pmb{h}}}}:B\rightarrow \limfunc{A}\left( {{\pmb{g}/\pmb{h}}}%
\right) $, $\limfunc{ad}\nolimits_{B,{{\pmb{h}}}}\left( \xi \right) \left( %
\left[ \nu \right] \right) =\left[ [\![\xi ,\nu ]\!]\right] $, $\xi \in
\Gamma \left( B\right) $, $\nu \in \Gamma \left( {{\pmb{g}}}\right) $, where 
$\limfunc{A}\left( {{\pmb{g}/\pmb{h}}}\right) $ is the (transitive) Lie
algebroid of ${{\pmb{g}/\pmb{h}}}$. Clearly, $\Psi \in \left( \Gamma \left(
\bigwedge^{k}\left( {{\pmb{g}/\pmb{h}}}\right) ^{\ast }\right) \right)
^{\Gamma \left( B\right) }$ if and only if $\left( \#_{B}\circ \xi \right) 
\hspace{-0.1cm}\left\langle \Psi ,\left[ \nu _{1}\right] \wedge \hspace{%
-0.05cm}\ldots \hspace{-0.05cm}\wedge \left[ \nu _{k}\right] \right\rangle 
\hspace{-0.05cm}=\hspace{-0.05cm}\sum_{j=1}^{k}\left( -1\right) ^{j-1}%
\hspace{-0.1cm}\left\langle \Psi ,\left[ [\![j\circ \xi ,\nu _{j}]\!]\right]
\wedge \left[ \nu _{1}\right] \wedge \hspace{-0.05cm}\ldots \hat{\jmath}%
\ldots \hspace{-0.05cm}\wedge \left[ \nu _{k}\right] \right\rangle $ for all 
$\xi \in \Gamma \left( B\right) $ and $\nu _{j}\in \Gamma \left( {{\pmb{g}}}%
\right) $ (see \cite{K1}). In the space $\left( \Gamma \left( \bigwedge
\left( {{\pmb{g}/\pmb{h}}}\right) ^{\ast }\right) \right) ^{\Gamma \left(
B\right) }$ of invariant cross-sections there exists a differential $\bar{%
\delta}$ defined by the formula%
\begin{equation*}
\hspace{-0.1cm}\left\langle \bar{\delta}\Psi ,\left[ \nu _{1}\right] \wedge
\ldots \wedge \left[ \nu _{k}\right] \right\rangle =\dsum\limits_{i<j}\left(
-1\right) ^{i+j+1}\hspace{-0.1cm}\left\langle \Psi ,\left[ [\![\nu _{i},\nu
_{j}]\!]\right] \wedge \left[ \nu _{1}\right] \wedge \ldots \hat{\imath}%
\ldots \hat{\jmath}\ldots \wedge \left[ \nu _{k}\right] \right\rangle ,
\end{equation*}%
(see \cite{K5}) and we obtain the cohomology algebra 
\begin{equation*}
\mathsf{H}^{\bullet }\left( {{\pmb{g}}},B\right) :=\mathsf{H}^{\bullet }%
{\LARGE ((}\Gamma {\LARGE (}\bigwedge \left( {{\pmb{g}/\pmb{h}}}\right)
^{\ast }{\Large ))}^{\Gamma \left( B\right) },\bar{\delta}{\Large )}.
\end{equation*}

\begin{theorem}
The homomorphism $\Delta _{\left( A,B,\nabla \right) }$ commutes with the
differentials $\bar{\delta}$ and $d_{L}$, where $d_{L}$ is the differential
operator in $\Omega \left( L\right) =\Gamma \left( \bigwedge L^{\ast
}\right) $.
\end{theorem}

\begin{proof}
Since the pullback of differential forms $\nabla ^{\ast }:\Omega \left(
L\right) \rightarrow \Omega \left( A\right) $ commutes with differentials $%
d_{L}$ and $d_{A}$, it is sufficient to show the commutativity of $\Delta
_{\left( A,B\right) }$ with differentials $\bar{\delta}$ and $d_{A}$. Let $%
\xi _{0},\,\ldots ,\xi _{k}\in \Gamma \left( A\right) $, and $\Psi $ be an
arbitrary invariant cross-section of the degree $k$. The curvature tensor $%
\Omega ^{j\circ \lambda }$ of the connection $j\circ \lambda $ takes values
in the bundle ${{\pmb{h}}}$. Thus, we see that%
\begin{equation*}
\Omega ^{j\circ \lambda }(\#_{A}\circ \xi _{i},\#_{A}\circ \xi _{j})=[\breve{%
\lambda}\circ \xi _{i},\breve{\lambda}\circ \xi _{j}]-[\hspace{-0.1cm}[\xi
_{i},\breve{\lambda}\circ \xi _{j}]\hspace{-0.1cm}]+[\hspace{-0.1cm}[\xi
_{j},\breve{\lambda}\circ \xi _{i}]\hspace{-0.1cm}]+\breve{\lambda}\circ
\lbrack \hspace{-0.1cm}[\xi _{i},\xi _{j}]\hspace{-0.1cm}]\in {{\pmb{h}.}}
\end{equation*}%
Therefore, using invariance of $\Psi $, we have%
\begin{align*}
& {\large (}d_{A}\circ \Delta _{\left( A,B\right) }{\large )}\left( \Psi
\right) \left( \xi _{0},\ldots ,\xi _{k}\right) \\
& =\underset{i<j}{\sum }\left( -1\right) ^{i+j+1}\,\langle \Psi ,[[\hspace{%
-0.1cm}[\breve{\lambda}\circ \xi _{i},\omega \circ \xi _{j}]\hspace{-0.1cm}%
]]\wedge \lbrack -\breve{\lambda}\circ \xi _{0}]\wedge \ldots \hat{\imath}%
\ldots \hat{\jmath}\ldots \wedge \lbrack -\breve{\lambda}\circ \xi
_{k}]\rangle \\
& -\underset{i<j}{\sum }\left( -1\right) ^{i+j}\hspace{-0.1cm}\langle \Psi ,%
\left[ \Omega ^{j\circ \lambda }\left( \#_{A}\circ \xi _{i},\#_{A}\circ \xi
_{j}\right) \right] \wedge \lbrack -\breve{\lambda}\circ \xi _{0}]\wedge
\ldots \hat{\imath}\ldots \hat{\jmath}\ldots \wedge \lbrack -\breve{\lambda}%
\circ \xi _{k}]\rangle \\
& =\langle \bar{\delta}\Psi ,[-\breve{\lambda}\circ \xi _{0}]\wedge ~\ldots
~\wedge \lbrack -\breve{\lambda}\circ \xi _{k}]\rangle \\
& ={\large (}\Delta _{\left( A,B\right) }\circ \bar{\delta}{\large )}\left(
\Psi \right) \left( \xi _{0},\ldots ,\xi _{k}\right) .
\end{align*}
\end{proof}

\begin{corollary}
$\Delta _{\left( A,B\right) }$ \emph{and }$\Delta _{\left( A,B,\nabla
\right) }$ \emph{induce the cohomology homomorphisms }$\ $%
\begin{equation*}
\Delta _{\left( A,B\right) \#}:\mathsf{H}^{\bullet }\left( {{\pmb{g}}}%
,B\right) \rightarrow \mathsf{H}^{\bullet }\left( A\right)
\end{equation*}%
\emph{and}%
\begin{equation*}
\Delta _{\left( A,B,\nabla \right) \#}:\mathsf{H}^{\bullet }\left( {{\pmb{g}}%
},B\right) \rightarrow \mathsf{H}^{\bullet }\left( L\right) .
\end{equation*}
\end{corollary}

The significance of $\Delta _{\left( A,B\right) \#}$ follows from the fact
that $\Delta _{\left( A,B,\nabla \right) \#}$, for every flat $L$-connection 
$\nabla :L\rightarrow A$, is factorized by $\Delta _{\left( A,B\right) \#}$: 
\begin{equation}
\begin{CD}
\Delta _{\left( A,B,\nabla \right) \#}:\mathsf{H}^{\bullet}\left( {{\pmb{g},B}}\right)      @>{ \Delta _{\left( A,B\right) \#}}>>
\mathsf{H}^{\bullet}\left(A\right)      @>{{\nabla}^{\#} }>> \mathsf
{H}^{\bullet}\left( L\right) .
\end{CD}
\label{deltacohomology}
\end{equation}

The map $\Delta _{\left( A,B,\nabla \right) \#}$ we will call the \textit{%
characteristic homomorphism }of the FS-Lie algebroid $\left( A,B,\nabla
\right) $. We call elements of a subalgebra $\func{Im}\Delta _{\left(
A,B,\nabla \right) \#}\subset \mathsf{H}^{\bullet }\left( L\right) $ the 
\textit{secondary} (\textit{exotic})\textit{\ characteristic classes} of
this algebroid. In particular, $\Delta _{\left( A,B\right) \#}=\Delta
_{\left( A,B,\func{id}_{A}\right) \#}$ is the characteristic homomorphism of
the Lie subalgebroid $B\subset A$. We define the last homomorphism as the 
\textit{universal exotic characteristic homomorphism }and the characteristic
classes from its image as the \textit{universal characteristic classes} of
the pair $B\subset A$.

\subsection{The Case for Regular Lie Algebroids and Usual Flat Connections 
\label{sec3.1}}

Given a regular Lie algebroid $\left( {A},[\![\cdot ,\cdot
]\!],\#_{A}\right) $ over a regular, foliated manifold $\left( M,F\right) $,
consider two geometric structures:

-- a flat connection $\omega :F\rightarrow A$,

-- a Lie{\Huge \ }subalgebroid $j:B\hookrightarrow A$ over $\left(
M,F\right) $.

Let $\breve{\omega}:A\rightarrow {{\pmb{g}}}$ be the connection form of $%
\omega $. Let us consider an auxiliary connection $\lambda :F\rightarrow B$,
its extension $j\circ \lambda $\ to $A$\ and let $\breve{\lambda}%
:A\rightarrow \pmb{g}$ be its connection form.\emph{\ }Since $i\circ \breve{%
\omega}+\omega \circ \#_{A}=\func{id}_{A}$, it follows that $i\circ \breve{%
\omega}\circ j\circ \lambda =-i\circ \breve{\lambda}\circ \omega $. Hence,
we conclude that%
\begin{align*}
{\large (}\Delta _{\left( A,B,\omega \right) }\Psi {\large )}_{|x}\left(
w_{1}\wedge \ldots \wedge w_{k}\right) & =\left\langle \Psi _{x},{\large [}-%
{\large (}\breve{\lambda}\circ \omega {\large )}\left( w_{1}\right) {\large ]%
}\wedge \ldots \wedge {\large [}-{\large (}\breve{\lambda}\circ \omega 
{\large )}\left( w_{k}\right) {\large ]}\right\rangle \\
& =\left\langle \Psi _{x},\left[ \breve{\omega}_{x}\left( \tilde{w}%
_{1}\right) \right] \wedge \ldots \wedge \left[ \breve{\omega}_{x}\left( 
\tilde{w}_{k}\right) \right] \right\rangle ,
\end{align*}%
where $\tilde{w}_{i}=\lambda \left( w_{i}\right) $. Since $\#_{B}\hspace{%
-0.05cm}\left( \tilde{w}_{i}\right) =w_{i}$,%
\begin{equation*}
\Delta _{\left( A,B,\omega \right) \#}:\mathsf{H}^{\bullet }\left( {{\pmb{g}}%
},B\right) \longrightarrow \mathsf{H}^{\bullet }\left( F\right)
\end{equation*}%
is the characteristic homomorphism for the regular flat Lie algebroid $%
\left( A,B,\omega \right) $, which was considered in \cite{K5}.

\subsection{The Particular Case: The Classical Kamber-Tondeur Homomorphism
and Universal Characteristic Homomorphism Factorizing the Classical One\label%
{universal}}

Consider{\huge \ }any $G$-principal fibre bundle $P$ over a smooth manifold $%
M$, a flat connection $\omega \subset TP$ in $P$ and a connected\ $H$%
-reduction $P^{\prime }\subset P$, where $H\subset G$ is a closed Lie
subgroup of $G$ (we do not assume either connectedness or compactness of $H$%
). Applying the Lie functor for principal fibre bundles{\Huge \ }we can
consider Lie algebroids $\limfunc{A}\left( P\right) \ $and $\limfunc{A}%
\left( P^{\prime }\right) $ as well as the induced flat connection $\omega
^{A}:TM\rightarrow \limfunc{A}\left( P\right) $ in the Lie algebroid $%
\limfunc{A}\left( P\right) $ and the secondary characteristic homomorphism%
\begin{equation*}
\Delta _{\left( \limfunc{A}\left( P\right) ,\limfunc{A}\left( P^{\prime
}\right) ,\omega ^{A}\right) \#}:\mathsf{H}^{\bullet }\left( {{\pmb{g},}}%
\limfunc{A}\left( P^{\prime }\right) \right) \longrightarrow \mathsf{H}%
_{dR}^{\bullet }\left( M\right) 
\end{equation*}%
for the FS-Lie algebroid $\left( \limfunc{A}\left( P\right) ,\limfunc{A}%
\left( P^{\prime }\right) ,\omega ^{A}\right) $. This homomorphism is
\textquotedblleft equivalent\textquotedblright\ to the standard classical
homomorphism on principal fibre bundles 
\begin{equation*}
\Delta _{\left( P,P^{\prime },\omega \right) \#}:\mathsf{H}^{\bullet }\left( 
\mathfrak{g},H\right) \longrightarrow \mathsf{H}_{dR}^{\bullet }\left(
M\right) 
\end{equation*}%
given by F.~Kamber and Ph.~Tondeur \cite{K-T3}, where $\mathsf{H}^{\bullet }(%
\mathfrak{g},H)$ --- called the \textit{relative Lie algebra cohomology of }$%
\left( \mathfrak{g},H\right) $ (see \cite{Chev-Eil}, \cite{K-T3}) --- is the
cohomology space of the complex $(\bigwedge \left( \mathfrak{g}/\mathfrak{h}%
\right) ^{\ast H},d^{H})$ where $\bigwedge \left( \mathfrak{g}/\mathfrak{h}%
\right) ^{\ast H}$ is the space of invariant elements with respect to the
adjoint representation of the Lie group $H$ (see \cite[3.27]{K-T3}) and the
differential $d^{H}$ is defined by the formula:%
\begin{equation*}
\hspace{-0.2cm}\left\langle d^{H}\hspace{-0.1cm}\left( \psi \right) ,\left[
v_{1}\right] \wedge \ldots \wedge \left[ v_{k}\right] \right\rangle =%
\underset{i<j}{\sum }\left( -1\right) ^{i+j}\left\langle \psi ,\left[ \left[
v_{i},v_{j}\right] \right] \wedge \left[ v_{1}\right] \wedge \ldots \hat{%
\imath}\ldots \hat{\jmath}\ldots \wedge \left[ v_{k}\right] \right\rangle 
\end{equation*}%
for $\psi \in \bigwedge^{k}\left( \mathfrak{g}/\mathfrak{h}\right) ^{\ast H}$
and $v_{i}\in \mathfrak{g}$. The equivalence of these two characteristic
homomorphisms lies in the fact that there exists an isomorphism of algebras $%
\kappa :\mathsf{H}^{\bullet }\left( \mathfrak{g}{,H}\right) \overset{\simeq }%
{\longrightarrow }\mathsf{H}^{\bullet }\left( {{\pmb{g},}}\limfunc{A}\left(
P^{\prime }\right) \right) $ such that%
\begin{equation}
\Delta _{\left( \limfunc{A}\left( P\right) ,\limfunc{A}\left( P^{\prime
}\right) ,\omega ^{A}\right) \#}\circ \kappa =\Delta _{\left( P,P^{\prime
},\omega \right) \#}  \label{equival}
\end{equation}%
(see \cite[Theorem 6.1]{K5}). Therefore, the obtained algebras of
characteristic classes are identical. We recall that the isomorphism $\kappa 
$ on the level of cochains is defined via the isomorphism $\tilde{\kappa}%
:\left( \bigwedge \left( \mathfrak{g}/\mathfrak{h}\right) ^{\ast }\right)
^{H}\rightarrow \left( \Gamma \left( \bigwedge \left( {{\pmb{g}/\pmb{h}}}%
\right) ^{\ast }\right) \right) ^{\Gamma \left( B\right) }$ given by $\tilde{%
\kappa}\left( \psi \right) \left( x\right) =\func{Ad}_{P^{\prime },{{\pmb{g}}%
}}^{\wedge }\left( z\right) \left( \psi \right) ,$ $z\in P_{|x}^{^{\prime }},
$ where the representation $\func{Ad}_{P^{\prime },{{\pmb{g}}}}^{\wedge }\ $%
of $P^{\prime }$ on $\bigwedge^{k}\left( {{\pmb{g}}}/{{\pmb{h}}}\right)
^{\star }$ is induced by $\func{Ad}_{P^{\prime },{{\pmb{g}}}}:P^{\prime
}\rightarrow L\left( {{\pmb{g}}}/{{\pmb{h}}}\right) ,\;z\mapsto \lbrack 
\overset{\wedge }{z}],$ and $\overset{\wedge }{z}:\mathfrak{g}\overset{\cong 
}{\rightarrow }{{\pmb{g}}}_{|x},$ $v\mapsto \left[ A_{z\star v}\right] $, ($%
A_{z}:G\rightarrow P,$ $a\mapsto za$). The fact that $\tilde{\kappa}$ is an
isomorphism is obtained by Proposition 5.5.3 from \cite{K1} (just here the
assumption that $P^{\prime }$\ is connected is needed). We recall briefly
the definition of the homomorphism{\huge \ }$\Delta _{\left( P,P^{\prime
},\omega ^{A}\right) \#}${\huge \ }and reasoning giving (\ref{equival}). Let 
$\breve{\omega}:TP\rightarrow \mathfrak{g}$ denote the connection form of $%
\omega $. There exists a homomorphism of $G$-$DG$-algebras $\breve{\omega}%
_{\wedge }:\bigwedge \mathfrak{g}^{\ast }\rightarrow \Omega \left( P\right) $
(in view of the flatness of $\omega $) induced by the algebraic connection $%
\breve{\omega}:\mathfrak{g}^{\ast }\rightarrow \Omega \left( P\right) ,\
\alpha \mapsto \alpha \breve{\omega}=\langle \alpha ,\breve{\omega}\rangle $%
. The homomorphism $\breve{\omega}_{\wedge }$ is given by the formula $%
\breve{\omega}_{\wedge }\left( \phi ^{k}\right) _{z}\left(
v_{1},...,v_{k}\right) =\langle \phi ;\omega _{z}\left( v_{1}\right) \wedge
...\wedge \omega _{z}\left( v_{k}\right) \rangle $ and can be restricted to $%
H$-basic elements $\breve{\omega}_{H}:\left( \bigwedge \mathfrak{g}^{\ast
}\right) _{H}\rightarrow \Omega \left( P\right) _{H}$. According to the
isomorphisms $\left( \bigwedge \mathfrak{g}^{\ast }\right) _{H}\cong
\bigwedge \left( \mathfrak{g}/\mathfrak{h}\right) ^{\ast H}$ and $\Omega
\left( P\right) _{H}\cong \Omega \left( P/H\right) $ it gives a $DG$%
-homomorphism of algebras $\breve{\omega}_{H}:\bigwedge \left( \mathfrak{g}/%
\mathfrak{h}\right) ^{\ast H}\rightarrow \Omega \left( P/H\right) $.
Composing it with $s^{\ast }:\Omega \left( P/H\right) \rightarrow \Omega
\left( M\right) $, where $s:M\rightarrow P/H$ is the cross-section
determined by the $H$-reduction $P^{\prime }$, we obtain a homomorphism of $%
DG$-algebras $\Delta _{P,P^{\prime },\omega }:\bigwedge \left( \mathfrak{g}/%
\mathfrak{h}\right) ^{\ast H}\overset{~\breve{\omega}_{H}}{\longrightarrow }%
\Omega \left( P/H\right) \overset{~s^{\ast }}{\longrightarrow }\Omega \left(
M\right) $. Passing to cohomology we obtain $\Delta _{\left( P,P^{\prime
},\omega \right) \#}$. Because of the algebraic formula for $\breve{\omega}%
_{\wedge }$ we see that this homomorphism on the level of forms is given by%
\begin{equation*}
\left( \Delta _{P,P^{\prime },\omega }\left( \psi \right) \right) _{x}\left(
w_{1}\wedge \ldots \wedge w_{k}\right) =\left\langle \psi ,\left[ \breve{%
\omega}_{z}\left( \tilde{w}_{1}\right) \right] \wedge \ldots \wedge \left[ 
\breve{\omega}_{z}\left( \tilde{w}_{k}\right) \right] \right\rangle ,
\end{equation*}%
where $z\in P_{|x}^{\prime }$, $w_{i}\in T_{x}M$, $\tilde{w}_{i}\in
T_{z}P^{\prime }$, $\pi _{\ast }^{\prime }\tilde{w}_{i}=w_{i}$, with $\pi
^{\prime }:P^{\prime }\rightarrow M$. Therefore, the equality (\ref{equival}%
) holds (for details see \cite[Th. 6.1]{K5}).

Using the universal \textit{\ exotic characteristic homomorphism }$\Delta
_{\left( \limfunc{A}\left( P\right) ,\limfunc{A}\left( P^{\prime }\right)
\right) \#}$ for the pair of transitive Lie algebroids $\left( \limfunc{A}%
\left( P\right) ,\limfunc{A}\left( P^{\prime }\right) \right) $, $\limfunc{A}%
\left( P^{\prime }\right) \subset \limfunc{A}\left( P\right) $, we can
define the \emph{universal\ exotic characteristic homomorphism} 
\begin{equation*}
\Delta _{\left( P,P^{\prime }\right) \#}:=\Delta _{\left( \limfunc{A}\left(
P\right) ,\limfunc{A}\left( P^{\prime }\right) \right) \#}\circ \kappa :%
\mathsf{H}^{\bullet }\left( \mathfrak{g},H\right) \longrightarrow \mathsf{H}%
^{\bullet }\left( \limfunc{A}\left( P\right) \right) \longrightarrow \mathsf{%
H}_{dR}^{r\bullet }\left( P\right)
\end{equation*}%
for the reduction of a principal fibre bundle $P^{\prime }\subset P$ (where $%
\mathsf{H}_{dR}^{r\bullet }\left( P\right) $ is the space of cohomology of
right-invariant differential forms on $P$; we recall that $\mathsf{H}%
_{dR}^{r\bullet }\left( P\right) :=\mathsf{H}^{\bullet }\left( {\Omega }%
^{r}\left( P\right) \right) \simeq \mathsf{H}_{dR}^{\bullet }\left( P\right) 
$ if $G$ is compact and connected).

\begin{theorem}
The homomorphism $\Delta _{\left( P,P^{\prime }\right) \#}$ on the level of
differential forms is given by the following formula: 
\begin{equation*}
\left( \Delta _{P,P^{\prime }}\psi \right) _{z}\hspace{-0.1cm}\left(
w_{1}\wedge \ldots \wedge w_{k}\right) =\left\langle \psi ,\left[ -\lambda
_{z}\left( w_{1}\right) \right] \wedge \ldots \wedge \left[ -\lambda
_{z}\left( w_{k}\right) \right] \right\rangle ,
\end{equation*}%
where $\lambda $ is the form of any connection on $P$ extending an arbitrary
connection on $P^{\prime }$.
\end{theorem}

The commutativity of the diagram 
\begin{center}
\begin{picture}(2200,850)
\setsqparms[1`1`0`1;1200`550]
\putsquare(0,150)[ \mathsf{H}^{\bullet }\left( \mathfrak{g},H\right) `\mathsf{H}_{dR}^{r\hspace{0.05cm}\bullet }\left( P\right)` \mathsf{H}^{\bullet }\left( {{\pmb{g},}}\limfunc{A}\left( P^{\prime }\right) \right) `\mathsf{H}^{\bullet }\left( \limfunc{A}\left( P\right) \right) ;  \Delta _{\left( P,P^{\prime }\right) \#} `  \kappa ` ` \Delta _{\left( \limfunc{A}\left( P\right) ,\limfunc{A}\left( P^{\prime }\right) \right) \#} ]
\put(80,420){\makebox(0,0){${\cong }$}}
\put(1180,230){\line(0,1){390}}
\put(1220,230){\line(0,1){390}}
\setsqparms[1`0`0`0;1000`550]
\putsquare(1360,150)[ `\mathsf{H}_{dR}^{\bullet }\left( P\right) ``\mathsf{H}_{dR}^{\bullet }\left( M\right) ;  `   ` ` \omega ^{A\#} ]
\put(1450,150){\vector(1,0){700}}
\put(1340,620){\vector(2,-1){820}}
\put(1530,420){ \makebox(0,0){{$\omega ^{\#}$}}}
\end{picture}
\end{center}
where $\omega ^{\#}$ on the level of right-invariant differential forms ${%
\Omega }^{r}\left( P\right) $ is given as the pullback of differential
forms:\ 
\begin{equation*}
\omega ^{\ast }:{\Omega }^{r}\left( P\right) \longrightarrow {\Omega }\left(
M\right) ,\ \ \omega ^{\ast }\left( \phi \right) _{x}\left( u_{1}\wedge
\ldots \wedge u_{k}\right) =\phi _{z}\left( \tilde{u}_{1}\wedge ~\ldots
~\wedge \tilde{u}_{k}\right) ,
\end{equation*}%
where $z\in P_{|x}$,\ $\tilde{u}_{i}$ is the $\omega $-horizontal lift of $%
u_{i}$, yields the following theorem.

\begin{theorem}
The homomorphism $\Delta _{\left( P,P^{\prime },\omega \right) \#}$\ is
factorized by $\Delta _{\left( P,P^{\prime }\right) \#}$, i.e. the diagram
below commutes 
\begin{center}
\settriparms[-1`1`1;]
\Atriangle[\mathsf{H}^{r \bullet}_{dR}(P)`\mathsf{H}^{\bullet}({\mathfrak{g}},H)`\mathsf{H}^{\bullet}_{dR}(M).;{\Delta_{(P,P^{\prime})\#}}`
{\omega}^{\#}`{\Delta_{(P,P^{\prime},\omega)\#}}]
\end{center}%
\end{theorem}

From the above we get a corollary on the existing of the new universal
exotic characteristic homomorphism for a $G$-principal fibre bundle $P$ and
its $H$-reduction $P^{\prime }$.

\begin{corollary}
\label{11 copy(1)}If $G$ is a compact, connected Lie group and $P^{\prime }$
is a connected $H$-reduction in a $G$-principal bundle $P$, $H\subset G$,
then there exists a homomorphism of algebras%
\begin{equation*}
\Delta _{\left( P,P^{\prime }\right) \#}:\mathsf{H}^{\bullet }\left( 
\mathfrak{g},H\right) \longrightarrow \mathsf{H}_{dR}^{\bullet }\left(
P\right)
\end{equation*}%
\emph{(}called a universal exotic characteristic homomorphism for the pair $%
P^{\prime }\subset P$\emph{) }such that for arbitrary flat connection $%
\omega $ in $P$, the characteristic homomorphism $\Delta _{\left(
P,P^{\prime },\omega \right) \#}:\mathsf{H}^{\bullet }\left( \mathfrak{g}%
,H\right) \rightarrow \mathsf{H}_{dR}^{\bullet }\left( M\right) $ is
factorized by $\Delta _{\left( P,P^{\prime }\right) \#}$, i.e. the following
diagram is commutative 
\begin{center}
\settriparms[-1`1`1;]
\Atriangle[\mathsf{H}^{\bullet}_{dR}(P)`\mathsf{H}^{\bullet}({\mathfrak{g}},H)`\mathsf{H}^{\bullet}_{dR}(M).;{\Delta_{(P,P^{\prime})\#}}`{\omega}^{\#}`{\Delta_{(P,P^{\prime},\omega)\#}}]
\end{center}%
\end{corollary}

\section{Comparison with the Crainic Framework\label{Crsubsection}}

\subsection{The Crainic Approach\label{Crainicclasses}}

We briefly explain the Crainic approach to characteristic classes of a
representation (\cite{Cr-1}). Primarily we notice that arbitrary
representation $\nabla _{\xi }\nu $ of an arbitrary Lie algebroid $L$\
(irregular in general) in a vector bundle $f\ $can be described by a
homomorphism of Lie algebroids $\nabla :L\rightarrow \limfunc{A}\left( 
\mathfrak{f}\right) $ (i.e. a flat $L$-connection in $\limfunc{A}\left( 
\mathfrak{f}\right) $). The Crainic classes of a representation $\nabla $
live in the cohomology algebra $\mathsf{H}^{\bullet }\left( L\right) $ of
the Lie algebroid $L$. In the simplest case of the trivial vector bundle $%
\mathfrak{f}=M\times V$ ($\dim V=n$) they are constructed as follows: For a
frame $\left\{ e_{1},\ldots ,e_{n}\right\} $ of $\mathfrak{f}$ we introduce
a matrix $\omega =[\omega _{j}^{i}]\in M_{n\times n}\left( \Omega \left(
L\right) \right) $ of $1$-forms on $L$ such that $\nabla _{\xi
}e_{j}=\sum_{i}\omega _{j}^{i}\left( \xi \right) \cdot e_{i}$\ for all $\xi
\in \Gamma \left( L\right) $. Clearly $\func{tr}\left( \omega \right) =\func{%
tr}\left( \tilde{\omega}\right) $, where $\tilde{\omega}=\frac{\omega
+\omega ^{T}}{2}$ is the symmetrization of $\omega $, and the flatness
condition implies that for natural numbers $k$,%
\begin{equation}
\func{tr}(\tilde{\omega}^{2k-1})  \label{trace}
\end{equation}%
are closed forms on $L$. Their cohomology classes are independent of the
choice of frames. These classes vanish if $\nabla $ is a Riemannian
connection with respect to some Riemannian metric $h$ in $\mathfrak{f}$. A
Riemannian connection is a connection in a Riemannian reduction $\limfunc{L}%
\left( \mathfrak{f},\left\{ h\right\} \right) $ of the principal fibre
bundle of frames $\limfunc{L}\left( \mathfrak{f}\right) $. For an arbitrary
vector bundle $\mathfrak{f}$ Crainic uses a local construction (a suitable
cocycle) and the \v{C}ech double complex $\check{C}^{\ast }\left( \mathcal{U}%
,C^{\ast }\left( L\right) \right) $ together with the Mayer-Vietoris
argument. For $L=TM$ the usual exotic characteristic classes of flat vector
bundles $\mathfrak{f}$ are obtained. An explicit formula for an arbitrary $L$%
-flat real vector bundle $\left( \mathfrak{f},\nabla \right) $ is based on
the observation that in a local orthonormal frame $\left\{ e_{1},\ldots
,e_{n}\right\} $ of $\left( \mathfrak{f},\left\{ h\right\} \right) $ the
symmetrization $\tilde{\omega}\ $of $\omega $ is equal to the matrix of the
symmetric-values form $\omega \left( \mathfrak{f},h\right) =\frac{1}{2}%
{\large (}\nabla -\nabla ^{h}{\large )}$, where $\nabla ^{h}$ is the adjoint 
$L$-connection induced by the metric $h$. The adjoint connection $\nabla
^{h} $ remains flat. The classes will be given as in (\ref{trace}), with $%
\tilde{\omega}$ replaced by $\omega \left( \mathfrak{f},h\right) $. One
explicit formula up to a constant uses the Chern-Simons transgression
differential forms $\limfunc{cs}\nolimits_{k}$ for suitable two connections
and is given by (see \cite{Cr-F})%
\begin{equation*}
u_{2k-1}\left( \mathfrak{f}\right) =\left[ u_{2k-1}\left( \mathfrak{f}%
,\nabla \right) \right] \in \mathsf{H}^{2k-1}\left( L\right) ,
\end{equation*}%
where 
\begin{equation}
u_{2k-1}\left( \mathfrak{f},\nabla \right) =\left( -1\right) ^{\frac{k+1}{2}}%
\limfunc{cs}\nolimits_{k}{\large (}\nabla ,\nabla ^{h}{\large )}  \label{u}
\end{equation}%
and $k$ is an odd natural ($u_{2k-1}\left( \mathfrak{f}\right) $ is trivial
if $k$ is even). We recall that 
\begin{eqnarray*}
\limfunc{cs}\nolimits_{k}{\large (}\nabla ,\nabla ^{h}{\large )}
&=&\dint_{0}^{1}\limfunc{ch}\nolimits_{k}{\large (}\nabla ^{\limfunc{aff}}%
{\large )}\in \Omega ^{2k-1}\left( L\right) , \\
\left( \int_{0}^{1}\limfunc{ch}\nolimits_{k}{\large (}\nabla ^{\limfunc{aff}}%
{\large )}\right) _{\xi _{1},\ldots ,\xi _{2k-1}} &=&\int_{0}^{1}\limfunc{ch}%
\nolimits_{k}(\nabla ^{\limfunc{aff}})_{\frac{\partial }{\partial t},\xi
_{1},\ldots ,\xi _{2k-1}}|_{\left( t,\bullet \right) }dt,\ \ \ \ \ \xi
_{i}\in \Gamma \left( L\right) ,
\end{eqnarray*}%
where $\nabla ^{\limfunc{aff}}=\left( 1-t\right) \cdot \tilde{\nabla}+t\cdot 
\tilde{\nabla}^{h}:T\mathbb{R}\times A\longrightarrow \limfunc{A}\left( 
\func{pr}_{2}^{\;\ast }\mathfrak{f}\right) $ is the affine combination of
the connections $\tilde{\nabla}$ and $\tilde{\nabla}^{h}$ whereas $\limfunc{%
ch}\nolimits_{k}\left( \nabla \right) =\limfunc{Tr}{\Large (}{\large (}%
R^{\nabla }{\large )}^{k}{\Large )}$. The connection $\tilde{\nabla}:T%
\mathbb{R}\times L\rightarrow A\left( \limfunc{pr}\nolimits_{2}^{\ast }%
\mathfrak{f}\right) $, $\tilde{\nabla}_{\left( v_{t},\xi _{x}\right) }\left(
\nu \circ \limfunc{pr}\nolimits_{2}\right) =\nabla _{\xi _{x}}\left( \nu
\right) ,$ is the lifting of the connection $\nabla $ through the projection 
$\limfunc{pr}_{2}:\mathbb{R}\times M\rightarrow M$. If $\nabla $ is flat,
then $\tilde{\nabla}$ is flat, too.

\subsection{The Secondary Characteristic Homomorphism for Riemannian
Reductions\label{s341}}

Let $\left( \mathfrak{f},\left\{ h\right\} \right) $ denote a vector bundle
of the rank $n$ over a manifold $M$ with a Riemannian metric $h$. The metric 
$h$ yields \cite{K3} the Lie subalgebroid $B=\limfunc{A}\left( \mathfrak{f}%
,\left\{ h\right\} \right) $ of the algebroid $\limfunc{A}\left( \mathfrak{f}%
\right) $ of the vector bundle $\mathfrak{f}$ and a$\ $reduction $\limfunc{L}%
\left( \mathfrak{f},\left\{ h\right\} \right) $ of the frames bundle $%
\limfunc{L}\mathfrak{f}$ of $\mathfrak{f}$; ${\large (}u:\mathbb{R}%
^{n}\rightarrow \mathfrak{f}_{|x}{\large )}\in \limfunc{L}\left( \mathfrak{f}%
,\left\{ h\right\} \right) $ if and only if $u$ is an isometry. Taking the
canonical isomorphism $\Phi _{\mathfrak{f}}:\limfunc{A}\left( \limfunc{L}%
\mathfrak{f}\right) \rightarrow \limfunc{A}\left( \mathfrak{f}\right) $ of
Lie algebroids (\cite{K1}) we have $\limfunc{A}\left( \mathfrak{f,}\left\{
h\right\} \right) =\Phi _{\mathfrak{f}}\left[ \limfunc{A}\left( \limfunc{L}%
\left( \mathfrak{f,}\left\{ h\right\} \right) \right) \right] $. We observe
that $\mathcal{\alpha }\in \Gamma \left( \limfunc{A}\left( \mathfrak{f}%
\right) \right) $ belongs to $\Gamma \left( \limfunc{A}\left( \mathfrak{f,}%
\left\{ h\right\} \right) \right) $ if and only if for any cross-sections $%
\nu ,\mu \in \Gamma \left( \mathfrak{f}\right) $ the formula $h\left( 
\mathcal{\alpha }\left( \nu \right) ,\mu \right) =\left( \#\alpha \right)
\left( h\left( \nu ,\mu \right) \right) -h\left( \nu ,\mathcal{\alpha }%
\left( \mu \right) \right) $ holds. The Atiyah sequences for $\limfunc{A}%
\left( \mathfrak{f}\right) $ and $\limfunc{A}\left( \mathfrak{f,}\left\{
h\right\} \right) $ are 
\begin{equation*}
0\longrightarrow \limfunc{End}\left( \mathfrak{f}\right) \overset{i}{%
\longrightarrow }\limfunc{A}\left( \mathfrak{f}\right) \overset{\pi }{%
\longrightarrow }TM\longrightarrow 0,
\end{equation*}%
\begin{equation*}
0\longrightarrow \limfunc{Sk}\left( \mathfrak{f}\right) \longrightarrow 
\limfunc{A}\left( \mathfrak{f,}\left\{ h\right\} \right) \longrightarrow
TM\longrightarrow 0,
\end{equation*}%
where $\limfunc{Sk}\left( \mathfrak{f}\right) \subset \limfunc{End}\left( 
\mathfrak{f}\right) $ is the vector subbundle of $h$-skew symmetric
endomorphisms. Let $L$ be a Lie algebroid over $M$ and $\nabla :L\rightarrow 
\limfunc{A}\left( \mathfrak{f}\right) $ any flat $L$-connection in $%
\mathfrak{f}$. Let us consider FS-Lie algebroids $\left( \left( \limfunc{A}%
\left( \mathfrak{f}\right) ,\limfunc{A}\left( \mathfrak{f},\left\{ h\right\}
\right) \right) ,\nabla \right) $ and $\left( \limfunc{A}\left( \mathfrak{f}%
\right) ,\limfunc{A}\left( \mathfrak{f},\left\{ h\right\} \right) ,\limfunc{%
id}\nolimits_{A\left( \mathfrak{f}\right) }\right) $ and theirs secondary
characteristic homomorphisms denote, for shortness, by%
\begin{equation*}
\Delta _{\#}:\mathsf{H}^{\bullet }\left( \limfunc{End}\mathfrak{f},\limfunc{A%
}\left( \mathfrak{f},\left\{ h\right\} \right) \right) \longrightarrow 
\mathsf{H}^{\bullet }\left( L\right) ,
\end{equation*}%
\begin{equation*}
\Delta _{o\#}:\mathsf{H}^{\bullet }\left( \limfunc{End}\mathfrak{f},\limfunc{%
A}\left( \mathfrak{f},\left\{ h\right\} \right) \right) \longrightarrow 
\mathsf{H}^{\bullet }\left( \limfunc{A}\left( \mathfrak{f}\right) \right) ,
\end{equation*}%
respectively and take into the consideration the isomorphism%
\begin{equation*}
\kappa :\mathsf{H}^{\bullet }\left( \mathfrak{gl}\left( n,\mathbb{R}\right)
,O\left( n\right) \right) \overset{\cong }{\longrightarrow }\mathsf{H}%
^{\bullet }\left( \limfunc{End}\mathfrak{f},\limfunc{A}\left( \mathfrak{f}%
,\left\{ h\right\} \right) \right)
\end{equation*}%
of algebras described in Section \ref{universal}. If the vector bundle $%
\mathfrak{f}$ is nonorientable, then%
\begin{equation*}
\mathsf{H}^{\bullet }\left( \limfunc{End}\mathfrak{f},\limfunc{A}\left( 
\mathfrak{f},\left\{ h\right\} \right) \right) \overset{\kappa }{\cong }%
\mathsf{H}^{\bullet }\left( \mathfrak{gl}\left( n,\mathbb{R}\right) ,O\left(
n\right) \right) \cong \bigwedge \left( y_{1},y_{3},\ldots ,y_{n^{\prime
}}\right) ,
\end{equation*}%
where $n^{\prime }$ is the largest odd integer $\leq n$ ($n^{\prime }=2\left[
\frac{n+1}{2}\right] -1$) and $y_{2k-1}\in \mathsf{H}^{4k-3}\left( \limfunc{%
End}\mathfrak{f},\limfunc{A}\left( \mathfrak{f},\left\{ h\right\} \right)
\right) $ is represented by the multilinear trace form $\widetilde{y}%
_{2k-1}\in \Gamma \left( \bigwedge\nolimits^{4k-3}\left( \limfunc{End}%
\mathfrak{f/}\limfunc{Sk}\mathfrak{f}\right) ^{\ast }\right) $, $k\in
\left\{ 1,2,\ldots ,\left[ \frac{n+1}{2}\right] \right\} $, see \cite[6.31,
p. 142]{K-T3}.

In the case of an oriented vector bundle $\mathfrak{f}$ with a volume form $%
\func{v}$, the metric $h$ and $\func{v}$ induce an $SO\left( n,\mathbb{R}%
\right) $-reduction $\limfunc{L}\left( \mathfrak{f},\left\{ h,\func{v}%
\right\} \right) $ of the frames bundle $\limfunc{L}\mathfrak{f}$ of $%
\mathfrak{f}$; $\ {\large (}u:\mathbb{R}^{n}\rightarrow \mathfrak{f}_{|x}%
{\large )}\in \limfunc{L}\left( \mathfrak{f},\left\{ h,\func{v}\right\}
\right) $ if and only if $u$ is an isometry keeping the orientation.
Clearly, $\limfunc{A}\left( \mathfrak{f,}\left\{ h,\func{v}\right\} \right) =%
\limfunc{A}\left( \mathfrak{f,}\left\{ h\right\} \right) $, and hence $%
\mathsf{H}^{\bullet }\left( \limfunc{End}\mathfrak{f},\limfunc{A}\left( 
\mathfrak{f},\left\{ h\right\} \right) \right) \cong \mathsf{H}^{\bullet
}\left( \limfunc{End}\mathfrak{f},\limfunc{A}\left( \mathfrak{f},\left\{ h,%
\func{v}\right\} \right) \right) $. If $\mathfrak{f}$ is orientable and odd
rank (see \cite{Godbillon}), 
\begin{equation*}
\mathsf{H}^{\bullet }\left( \mathfrak{gl}\left( n,\mathbb{R}\right)
,SO\left( n\right) \right) \cong \mathsf{H}^{\bullet }\left( \mathfrak{gl}%
\left( n,\mathbb{R}\right) ,O\left( n\right) \right) .
\end{equation*}%
Using $\kappa $ and the above, we have%
\begin{equation*}
\mathsf{H}^{\bullet }\left( \limfunc{End}\mathfrak{f},\limfunc{A}\left( 
\mathfrak{f},\left\{ h,\func{v}\right\} \right) \right) \cong \mathsf{H}%
^{\bullet }\left( \mathfrak{gl}\left( n,\mathbb{R}\right) ,O\left( n\right)
\right) \cong \bigwedge \left( y_{1},y_{3},\ldots ,y_{n}\right) .
\end{equation*}%
If the vector bundle $\mathfrak{f}$ is orientable of even rank $n=2m$, then%
\begin{equation*}
\mathsf{H}^{\bullet }\left( \func{End}\mathfrak{f},\limfunc{A}\left( 
\mathfrak{f},\left\{ h,\func{v}\right\} \right) \right) \cong \mathsf{H}%
^{\bullet }\left( \mathfrak{gl}\left( 2m,\mathbb{R}\right) ,SO\left(
2m\right) \right) \cong \bigwedge \left( y_{1},y_{3},\ldots
,y_{2m-1},y_{2m}\right) ,
\end{equation*}%
where $y_{2k-1}\in \mathsf{H}^{4k-3}\left( \limfunc{End}\mathfrak{f},%
\limfunc{A}\left( \mathfrak{f},\left\{ h,\func{v}\right\} \right) \right) $
are defined as above and $y_{2m}\in \mathsf{H}^{2m}\left( \mathfrak{gl}%
\left( n,\mathbb{R}\right) ,SO\left( n\right) \right) \cong \mathsf{H}%
^{2m}\left( \limfunc{End}\mathfrak{f},\limfunc{A}\left( \mathfrak{f},\left\{
h,\func{v}\right\} \right) \right) $ is represented by $\widetilde{y}%
_{2m}\in \Gamma \left( \bigwedge^{2m}\left( \limfunc{End}\mathfrak{f}{/}%
\limfunc{Sk}\mathfrak{f}\right) ^{\ast }\right) $,\ 
\begin{equation}
\widetilde{y}_{2m}\left( \left[ A_{1}\right] ,\ldots ,\left[ A_{2m}\right]
\right) =d\left( z_{2m-1}\right) (\widetilde{A\,}_{1},\ldots ,\widetilde{A\,}%
_{2m}),  \label{d}
\end{equation}%
$A_{1},\ldots ,A_{2m}\in \Gamma \left( \func{End}\mathfrak{f}\right) $, and
where $\widetilde{A\,}_{k}$ denotes the symmetrization of $A_{k}$, $d$ is
the usual differential on the algebra $\bigwedge \left( \limfunc{End}%
\mathfrak{f}\right) ^{\ast }$, $z_{2m-1}\in \Gamma \left(
\bigwedge^{2m-1}\left( \limfunc{End}\mathfrak{f}\right) ^{\ast }\right) $\
is given by%
\begin{eqnarray*}
&&\hspace{-0.3cm}z_{2m-1}\left( A_{1},\ldots ,A_{2m-1}\right) \\
&=&c\left( m\right) \hspace{-0.3cm}\dsum\limits_{\sigma \in S_{2m-1}}\hspace{%
-0.3cm}\func{sgn}\sigma \left( e,\alpha A_{\sigma _{1}}\wedge \alpha \left[
A_{\sigma _{2}},A_{\sigma _{3}}\right] \wedge \ldots \wedge \alpha \left[
A_{\sigma _{2m-2}},A_{\sigma _{2m-1}}\right] \right) ,
\end{eqnarray*}%
where $c\left( m\right) =\frac{\left( -1\right) ^{m-1}\left( m-1\right) !}{%
2^{m-1}\left( 2m-1\right) !}\in \mathbb{R}$, $e$ is a non-zero cross-section
of $\dbigwedge\nolimits^{2m}\mathfrak{f}$ and \ $\alpha :\func{End}\mathfrak{%
f}\rightarrow \bigwedge^{2}\mathfrak{f}$ is given by $\left( \alpha \left(
A\right) ,\nu \wedge \mu \right) =\frac{1}{2}\left( \left( A\nu ,\mu \right)
-\left( \nu ,A\mu \right) \right) $, $A\in \Gamma \left( \func{End}\mathfrak{%
f}\right) $,\ $\nu ,\,\mu \in \Gamma \left( \mathfrak{f}\right) $. We add
that $z_{2m-1}$ is the image of the Pfaffian for a pair $\left( \mathfrak{f}%
,e\right) $ by the Cartan map for $\func{End}\mathfrak{f}$ (for the Cartan
map see for example \cite[Ch. VI, 6.7, 6.8]{GHV}).

We shall show that $\Delta _{\#}\left( y_{2j-1}\right) $ is (up to a
constant) equal to the Crainic class $u_{4j-3}\left( \mathfrak{f}\right) $
for all $j\in \left\{ 1,2,\ldots ,\left[ \frac{n+1}{2}\right] \right\} $. Let%
{\Huge \ }$\nabla _{0}$,$\nabla _{1}:L\rightarrow \limfunc{A}\left( 
\mathfrak{f}\right) $ be arbitrary\ two $L$-connections in $\mathfrak{f}$%
{\Huge \ }and let $\nabla ^{\limfunc{aff}}=\left( 1-t\right) \tilde{\nabla}%
_{0}+t\tilde{\nabla}_{1}:T\mathbb{R}\times L\rightarrow A\left( \limfunc{pr}%
\nolimits_{2}^{\ast }\mathfrak{f}\right) $ be their affine combination.
Observe%
\begin{equation}
R^{\nabla _{1}}=R^{\nabla _{0}}+d^{\nabla _{0}}\theta +\left[ \theta ,\theta %
\right] ,  \label{(6)}
\end{equation}%
where 
\begin{equation}
\theta =\nabla _{1}-\nabla _{0}\in \Omega ^{1}\left( L;\limfunc{End}%
\mathfrak{f}\right)  \label{(66)}
\end{equation}%
and $\left[ \theta ,\theta \right] =\theta ^{2}=\theta \wedge \theta \in
\Omega ^{2}\left( L;\limfunc{End}\mathfrak{f}\right) $, $\left[ \theta
,\theta \right] \left( \xi ,\eta \right) =\left[ \theta \left( \xi \right)
,\theta \left( \eta \right) \right] $. The 1-form $\theta $\ we can lifted to%
{\huge \ }$\tilde{\theta}\in \Omega ^{1}\left( T\mathbb{R}\times L;\limfunc{%
End}\mathfrak{f}\right) $ putting $\tilde{\theta}_{\left( v_{t},\xi
_{x}\right) }=\theta _{\xi _{x}}$. The cross-section $\left( 0,\xi \right) $
of $T\mathbb{R}\times L$ will be denoted by $\xi $ and $(\frac{\partial }{%
\partial t},0)$ by $\frac{\partial }{\partial t}$. Observe that $\nabla ^{%
\limfunc{aff}}=\tilde{\nabla}_{0}+\Xi $, where\ $\Xi _{|\left( t,x\right)
}=t\cdot \tilde{\theta}_{x}$ and $(d^{\tilde{\nabla}_{1}}\tilde{\theta}%
)_{\xi ,\eta }\left( \nu \circ \limfunc{pr}\nolimits_{2}\right) =\left(
d^{\nabla _{1}}\theta \right) _{\xi ,\eta }\left( \nu \right) \circ \limfunc{%
pr}\nolimits_{2}$ for any $\xi ,\eta \in \Gamma \left( L\right) $, $x\in M$, 
$t\in \mathbb{R}$. The affine combination $\nabla ^{\limfunc{aff}}$ of flat
connections cannot be flat even if $\nabla _{0}$ is flat:\ by (\ref{(6)})
for flat $\nabla _{0}$ we have%
\begin{equation}
R^{\nabla ^{\limfunc{aff}}}=d^{\tilde{\nabla}_{0}}\Xi +\left[ \Xi ,\Xi %
\right] .  \label{(9)}
\end{equation}

\begin{lemma}
The curvature tensor $R^{\nabla ^{\limfunc{aff}}}$ of the affine combination 
$\nabla ^{\limfunc{aff}}$ of two flat $L$-connections $\nabla _{0}$%
,\thinspace $\nabla _{1}$ has the following properties 
\begin{equation}
(R^{\nabla ^{\limfunc{aff}}})_{\frac{\partial }{\partial t},\xi }\left( \nu
\circ \limfunc{pr}\nolimits_{2}\right) =\theta _{\xi }\left( \nu \right) ,
\label{(11)}
\end{equation}%
\begin{equation}
(R^{\nabla ^{\limfunc{aff}}})_{\xi ,\eta }\left( \nu \circ \limfunc{pr}%
\nolimits_{2}\right) _{|\left( t,\cdot \right) }=\left( t^{2}-t\right) \cdot
\left( \theta \wedge \theta \right) _{\xi ,\eta }\left( \nu \right) ,
\label{(10)}
\end{equation}%
\begin{equation}
((R^{\nabla ^{\limfunc{aff}}})_{\frac{\partial }{\partial t},\xi _{1},\ldots
,\xi _{2k-1}}^{k})_{|_{\left( t,\cdot \right) }}=k\cdot t^{k-1}\cdot \left(
t-1\right) ^{k-1}\cdot \theta _{\xi _{1},\ldots ,\xi _{2k-1}}^{2k-1}.
\label{(12)}
\end{equation}
\end{lemma}

\begin{proof}
Formula (\ref{(11)}) is clear. To see (\ref{(10)}), we need only to observe
(because of (\ref{(9)}), (\ref{(6)}) and the flatness of $\nabla _{0}$ and $%
\nabla _{1}$) that 
\begin{align*}
(d^{\tilde{\nabla}_{0}}\Xi )_{\xi ,\eta }\left( \nu \circ \limfunc{pr}%
\nolimits_{2}\right) _{|\left( t,\cdot \right) }& =t\cdot R_{\xi ,\eta
}^{\nabla _{1}}\left( \nu \right) -t\cdot \left( \theta \wedge \theta
\right) _{\xi ,\eta }\left( \nu \right) =-t\cdot \left( \theta \wedge \theta
\right) _{\xi ,\eta }\left( \nu \right) , \\
\left[ \Xi ,\Xi \right] _{\xi ,\eta }\left( \nu \circ \limfunc{pr}%
\nolimits_{2}\right) _{|\left( t,\cdot \right) }& =t^{2}\cdot \left( \theta
\wedge \theta \right) _{\xi ,\eta }\left( \nu \right) .
\end{align*}%
Formula (\ref{(12)}) can be proved by induction with respect to $k.$ Indeed,
from (\ref{(11)}) we have the step $k=1$. Let $n\in \mathbb{N}$. Assume that
(\ref{(12)}) holds for all $k\leq n$. From this, (\ref{(10)}), (\ref{(11)})
and the associativity of the algebra $\left( \limfunc{End}\mathfrak{%
f\,,\circ }\right) $ we\ get%
\begin{align*}
& ((R^{\nabla ^{\limfunc{aff}}})_{\frac{\partial }{\partial t},\xi
_{1},....,\xi _{2n+1}}^{n+1})_{|\left( t,\cdot \right) } \\
& =(((R^{\nabla ^{\limfunc{aff}}})^{n}\wedge R^{\nabla ^{\limfunc{aff}}})_{%
\frac{\partial }{\partial t},\xi _{1},....,\xi _{2n+1}})_{|\left( t,\cdot
\right) } \\
& =\sum_{\sigma \in S\left( 2n-1,2\right) }\func{sgn}\sigma \cdot
((R^{\nabla ^{\limfunc{aff}}})_{\frac{\partial }{\partial t},\xi _{\sigma
_{1}},....,\xi _{\sigma _{2n-1}}}^{n}\circ R_{\xi _{\sigma _{2n}},\xi
_{\sigma _{2n+1}}}^{\nabla ^{\limfunc{aff}}})_{|\left( t,\cdot \right) } \\
& \ \ \ +\sum_{\sigma \in S\left( 2n,1\right) }\func{sgn}\sigma \cdot
((R^{\nabla ^{\limfunc{aff}}})_{\xi _{\sigma _{1}},....,\xi _{\sigma
_{2n}}}^{n}\circ R_{\frac{\partial }{\partial t},\xi _{\sigma
_{2n+1}}}^{\nabla ^{\limfunc{aff}}})_{|\left( t,\cdot \right) } \\
& =\sum_{\sigma \in S\left( 2n-1,2\right) }\func{sgn}\sigma \cdot (n\cdot
t^{n}\left( t-1\right) ^{n}\cdot \theta _{\xi _{\sigma _{1}},....,\xi
_{\sigma _{2n-1}}}^{2n-1}\circ \left( \theta \wedge \theta \right) _{\xi
_{\sigma _{2n}},\xi _{\sigma _{2n+1}}}) \\
& \ \ \ +\sum_{\sigma \in S\left( 2n,1\right) }\func{sgn}\sigma \cdot
(t^{n}\left( t-1\right) ^{n}\cdot \theta _{\left( a_{\sigma
_{1}},....,a_{\sigma _{2n}}\right) }^{2n}\circ \theta _{a_{\sigma _{2n+1}}})
\\
& =n\,t^{n}\left( t-1\right) ^{n}\cdot \theta _{\xi _{1},....,\xi
_{2n+1}}^{2n+1}+t^{n}\left( t-1\right) ^{n}\cdot \theta _{\xi _{1},....,\xi
_{2n+1}}^{2n+1} \\
& =\left( n+1\right) t^{\left( n+1\right) -1}\left( t-1\right) ^{\left(
n+1\right) -1}\cdot \theta _{\xi _{1},....,\xi _{2n+1}}^{2\left( n+1\right)
-1}.
\end{align*}
\end{proof}

From the above we have the following theorem.

\begin{theorem}
The Chern-Simons transgression differential form $cs_{k}\left( \nabla
_{0},\nabla _{1}\right) $ for two flat $L$-connections $\nabla _{0}$%
,\thinspace $\nabla _{1}$, is equal to 
\begin{equation}
\limfunc{cs}\nolimits_{k}\left( \nabla _{0},\nabla _{1}\right) =\left(
-1\right) ^{k+1}\,\frac{\,k!\cdot \left( k-1\right) !\,}{\left( 2k-1\right) !%
}\cdot \func{tr}\theta ^{2k-1},  \label{(13)}
\end{equation}%
where $\theta $ is defined by \emph{(\ref{(66)})}.
\end{theorem}

The above formula is well known in the classical cases on principal fibre
bundles (see, for example, the papers by Chern and Simons \cite{Chern-Simons}%
, Heitsch and Lawson \cite{Heitsch-Lawson}, 1974).

Let $\widetilde{\ \ }:\func{End}\mathfrak{f}\rightarrow \func{End}\mathfrak{f%
}$, $v\mapsto \widetilde{v}:=\frac{1}{2}\left( v+v^{\ast }\right) $ denote
the symmetrization. Let us consider $\func{id}_{\limfunc{A}\left( \mathfrak{f%
}\right) }$ as an $\limfunc{A}\left( \mathfrak{f}\right) $-connection in $%
\mathfrak{f}$ and take its adjoint $\func{id}_{\limfunc{A}\left( \mathfrak{f}%
\right) }^{h}$ induced by the metric $h$. Let $\lambda :TM\rightarrow 
\limfunc{A}\left( \mathfrak{f}\right) $ be any $h$-Riemannian connection
(i.e. a connection such that $\func{Im}\lambda \subset \limfunc{A}\left( 
\mathfrak{f},h\right) $, or equivalently $\lambda ^{h}=\lambda $) and $%
\breve{\lambda}:\limfunc{A}\left( \mathfrak{f}\right) \rightarrow \func{End}%
\mathfrak{f}$ be its connection form. Since $i\circ \breve{\lambda}+\lambda
\circ \pi =\func{id}_{\limfunc{A}\left( \mathfrak{f}\right) }$, for any
cross-section $\alpha $ of $\limfunc{A}\left( \mathfrak{f}\right) $ we have%
\begin{equation}
-\widetilde{\breve{\lambda}\left( \alpha \right) }=\frac{1}{2}(\func{id}%
_{A\left( \mathfrak{f}\right) }^{h}-\func{id}_{\limfunc{A}\left( \mathfrak{f}%
\right) })\left( \alpha \right) .  \label{(3)}
\end{equation}%
Using (\ref{(3)}), (\ref{(13)}) and (\ref{u}) we get that%
\begin{equation*}
\Delta _{o}\left( \widetilde{y}_{2k-1}\right) =\left( -1\right) ^{k}\cdot
2^{3-4k}\cdot \frac{\left( 4k-3\right) !}{\left( 2k-1\right) !\cdot \left(
2k-2\right) !}\cdot u_{4k-3}{\large (}\mathfrak{f,}\func{id}_{A\left( 
\mathfrak{f}\right) }{\large )}.
\end{equation*}%
Since $\limfunc{cs}\nolimits_{k}(\nabla ,\nabla ^{h})=\nabla ^{\ast }(%
\limfunc{cs}\nolimits_{k}(\func{id}_{A\left( \mathfrak{f}\right) },\func{id}%
_{A\left( \mathfrak{f}\right) }^{h}))$ and $u_{4k-3}\left( \mathfrak{f,}%
\nabla \right) =\nabla ^{\#}u_{4k-3}{\large (}\mathfrak{f,}\func{id}%
_{A\left( \mathfrak{f}\right) }{\large )}$,%
\begin{equation*}
\Delta _{\#}\left( y_{2k-1}\right) =\left[ \nabla ^{\ast }\Delta _{o}\left( 
\widetilde{y}_{2k-1}\right) \right] =\frac{\left( -1\right) ^{k}\cdot \left(
4k-3\right) !}{2^{4k-3}\cdot \left( 2k-1\right) !\cdot \left( 2k-2\right) !}%
\cdot u_{4k-3}\left( \mathfrak{f}\right) .
\end{equation*}%
From the above formulae we can explain the relation between the
characteristic homomorphism $\Delta _{\#}:\mathsf{H}^{\bullet }\left( 
\limfunc{End}\mathfrak{f},A\left( \mathfrak{f},\left\{ h\right\} \right)
\right) \rightarrow \mathsf{H}^{\bullet }\left( L\right) $ of the triple $%
\left( A\left( \mathfrak{f}\right) ,A\left( \mathfrak{f},\left\{ h\right\}
\right) ,\nabla \right) $ and the family of the Crainic classes $\left\{
u_{4k-3}\left( \mathfrak{f}\right) \right\} $.

\begin{theorem}
Let $\mathfrak{f}$ be a real vector bundle over a manifold $M$ and 
\begin{equation*}
\Delta _{\#}:\mathsf{H}^{\bullet }\left( \limfunc{End}\mathfrak{f},A\left( 
\mathfrak{f},\left\{ h\right\} \right) \right) \longrightarrow \mathsf{H}%
^{\bullet }\left( L\right)
\end{equation*}%
the secondary characteristic homomorphism corresponding to $\left( A\left( 
\mathfrak{f}\right) ,A\left( \mathfrak{f},\left\{ h\right\} \right) ,\nabla
\right) $, where $\nabla :L\rightarrow A\left( \mathfrak{f}\right) $ is a
flat $L$-connection in $A\left( \mathfrak{f}\right) $.

\begin{itemize}
\item[(a)] If the vector bundle $\mathfrak{f}$ is nonorientable or
orientable and of odd rank $n$, then the image of $\Delta _{\#}$ is
generated by $u_{1}\left( \mathfrak{f}\right) $, $u_{5}\left( \mathfrak{f}%
\right) $,$\ldots $,$u_{4\left[ \frac{n+3}{4}\right] -3}\left( \mathfrak{f}%
\right) $.

\item[(b)] If the vector bundle $\mathfrak{f}$ is orientable and of even
rank $n=2m$, then the image of $\Delta _{\#}$ is generated by $u_{1}\left( 
\mathfrak{f}\right) $, $u_{5}\left( \mathfrak{f}\right) $,$\ldots $,$u_{4%
\left[ \frac{n+3}{4}\right] -3}\left( \mathfrak{f}\right) \ $and
additionally by $\Delta _{\#}\left( y_{2m}\right) $, where $y_{2m}$ is given
in \emph{(\ref{d})}.
\end{itemize}
\end{theorem}

\subsection{Example of a Nontrivial Universal Characteristic Class
Determined by the Pfaffian\label{Example_proof}}

Let $M$ be an oriented, connected manifold, $\dim M\geq 1$, and $\mathfrak{g}%
=\limfunc{End}\left( \mathbb{R}^{2}\right) $\emph{. }Given a transitive Lie
algebroid\emph{\ }$\left( {A},[\![\cdot ,\cdot ]\!],\#_{A}\right) $\emph{\ }%
over $M$, where $A=TM\oplus \limfunc{End}\left( \mathbb{R}^{2}\right) \cong
A\left( M\times \mathbb{R}^{2}\right) $ and\emph{\ }$\#_{A}=\limfunc{pr}%
\nolimits_{1}$\emph{\ }is a projection on the first factor, and%
\begin{equation*}
\lbrack \![\left( X_{1},\sigma _{1}\right) ,\left( X_{2},\sigma _{2}\right)
]\!]=\left( \left[ X_{1},X_{2}\right] ,X_{1}\left( \sigma _{2}\right)
-X_{2}\left( \sigma _{1}\right) +\left[ \sigma _{1},\sigma _{2}\right]
\right)
\end{equation*}%
for all $X_{1},X_{2}\in \mathfrak{X}\left( M\right) $, $\sigma _{1},\sigma
_{2}\in C^{\infty }\left( M;\limfunc{End}\left( \mathbb{R}^{2}\right)
\right) $, we have the Atiyah sequence%
\begin{equation*}
0\longrightarrow M\times \limfunc{End}\left( \mathbb{R}^{2}\right) \cong 
\limfunc{End}\left( M\times \mathbb{R}^{2}\right) {~}\overset{i}{{%
\hookrightarrow }}~A\overset{\limfunc{pr}\nolimits_{1}}{\longrightarrow }%
TM\longrightarrow 0.
\end{equation*}%
Let $B\subset A$ be the Riemannian reduction of $A$, i.e. $B=TM\oplus 
\limfunc{Sk}\left( \mathbb{R}^{2}\right) $ is a transitive subalgebroid of $%
A $. Observe that in the domain of the universal characteristic homomorphism 
$\Delta _{o\,\#}:\mathsf{H}^{\bullet }\left( M\times \mathfrak{g}{,B}\right)
\rightarrow \mathsf{H}^{\bullet }\left( A\right) $ is\emph{\ }$\left[ 
\widetilde{y}_{2}\right] \in \mathsf{H}^{2}\left( M\times \mathfrak{g}{,B}%
\right) $, where $\widetilde{y}_{2}\left( \left[ \sigma _{1}\right] ,\left[
\sigma _{2}\right] \right) =\limfunc{Pf}\left( \left[ \left[ \widetilde{%
\sigma }_{1},\widetilde{\sigma }_{2}\right] \right] \right) \ $for all $%
\sigma _{1},\sigma _{2}\in \Gamma \left( \ker \#_{A}\right) \cong C^{\infty
}\left( M;\mathfrak{g}\right) $. $\Delta _{o\,\#}\left( \left[ \widetilde{y}%
_{2}\right] \right) \in \mathsf{H}^{2}\left( A\right) $\emph{\ }is
represented by $\Delta _{o}\left( \widetilde{y}_{2}\right) \in \Omega
^{1}\left( A\right) $ given by%
\begin{equation*}
\Delta _{o}\left( \widetilde{y}_{2}\right) \left( \left( X_{1},\sigma
_{1}\right) ,\left( X_{2},\sigma _{2}\right) \right) =\limfunc{Pf}\left( %
\left[ \widetilde{\sigma }_{1},\widetilde{\sigma }_{2}\right] \right) .
\end{equation*}

\begin{theorem}
$\Delta _{o\,\#}\left( \widetilde{y}_{2}\right) \neq 0.$
\end{theorem}

\begin{proof}
Suppose that $\Delta _{o}\left( \widetilde{y}_{2}\right) $ is exact. Let $%
\Delta _{o}\left( \widetilde{y}_{2}\right) =d_{A}\left( \zeta \right) $ for
some $\zeta \in \Omega ^{1}\left( A\right) $. Thus we get that for all $%
\left( X_{1},\sigma _{1}\right) ,\left( X_{2},\sigma _{2}\right) \in 
\mathfrak{X}\left( M\right) \times C^{\infty }\left( M;\mathfrak{g}\right) $%
, $\limfunc{Pf}\left( \left[ \widetilde{\sigma }_{1},\widetilde{\sigma }_{2}%
\right] \right) $ is equal to%
\begin{equation*}
X_{1}\left( \zeta \left( X_{2},\sigma _{2}\right) \right) -X_{2}\left( \zeta
\left( X_{1},\sigma _{1}\right) \right) +\zeta \left( \left[ X_{1},X_{2}%
\right] ,X_{1}\left( \sigma _{2}\right) -X_{2}\left( \sigma _{1}\right) +%
\left[ \sigma _{1},\sigma _{2}\right] \right) .
\end{equation*}%
Observe that $\zeta =1\otimes \zeta _{1}+\zeta _{2}\otimes 1$ for some $%
\zeta _{1}\in \Gamma \left( M\times \mathfrak{g}^{\ast }\right) $ and $\zeta
_{2}\in \Omega ^{1}\left( M\right) $. For this reason, for $\sigma
_{1}=\sigma _{2}=0$, we obtain that $d_{dR}\left( \zeta _{2}\right) =0$.
Moreover, for $X_{1}=0$ and $\sigma _{2}=0$ we have%
\begin{equation}
X_{2}\left( \zeta _{1}\left( \sigma _{1}\right) \right) =-\zeta _{1}\left(
X_{2}\left( \sigma _{1}\right) \right)  \label{wlasnoscomega1}
\end{equation}%
for all\emph{\ }$X_{2}\in \mathfrak{X}\left( M\right) $, $\sigma _{1}\in
C^{\infty }\left( M;\mathfrak{g}\right) $. Let $\left\{
E_{1},E_{2},E_{3},E_{4}\right\} $\emph{\ }be a base of $\mathfrak{g}$. Fix $%
X\in \mathfrak{X}\left( M\right) $\emph{, }$\sigma \in C^{\infty }\left( M;%
\mathfrak{g}\right) $, and let $\zeta _{1}=\tsum\nolimits_{j}\zeta
_{1}^{j}E_{j}$\emph{\ }for some $\zeta _{1}^{j}\in C^{\infty }\left(
M\right) $. Note that $X\left( \zeta _{1}\left( \sigma \right) \right)
=\zeta _{1}\left( X\left( \sigma \right) \right) +X\left( \zeta _{1}\right)
\left( \sigma \right) $. Combining this with (\ref{wlasnoscomega1}) we
deduce that 
\begin{equation}
2\zeta _{1}\left( X\left( \sigma \right) \right) +X\left( \zeta _{1}\right)
\left( \sigma \right) =0.  \label{wlasnoscomega1_2}
\end{equation}%
Taking in (\ref{wlasnoscomega1_2}) constant functions $\sigma _{j}=1\cdot
E_{j}\in C^{\infty }\left( M;\mathfrak{g}\right) $, $j\in \left\{
1,2,3,4\right\} $, we see that $X\left( \zeta _{1}\right) =0$\emph{\ }for
all $X\in \mathfrak{X}\left( M\right) $. It follows that $\zeta _{1}$ is a
constant function. Let $\widetilde{\sigma }_{j}=\sigma ^{j}E_{j}\in
C^{\infty }\left( M;\mathfrak{g}\right) $ for some non--constant functions%
\emph{\ }$\sigma ^{j}\in C^{\infty }\left( M;\mathfrak{g}\right) $. (\ref%
{wlasnoscomega1_2}) now implies $X\left( \sigma ^{j}\right) \zeta _{1}^{j}=0$
for all $X\in \mathfrak{X}\left( M\right) $ and $j\in \left\{
1,2,3,4\right\} $. Hence $\zeta _{1}^{j}=0$. Since $\zeta _{1}=0$ and $%
d_{dR}\left( \zeta _{2}\right) =0$, $\limfunc{Pf}\left( \left[ \widetilde{%
\sigma }_{1},\widetilde{\sigma }_{2}\right] \right) =0$ for all $\sigma
_{1},\sigma _{2}\in C^{\infty }\left( M;\mathfrak{g}\right) $. On the second
hand $\limfunc{Pf}\left( \left[ \widetilde{\sigma }_{1},\widetilde{\sigma }%
_{2}\right] \right) $ is not a zero function for all $\sigma _{1},\sigma
_{2} $. Indeed, let $\left\{ e_{1},e_{2}\right\} $ be a base of $\mathbb{R}%
^{2}$ and $\left\{ e^{\ast 1},e^{\ast 2}\right\} $ the associated dual base
of\emph{\ }$\left( \mathbb{R}^{2}\right) ^{\ast }$. Let\emph{\ }$%
E_{1},E_{2},E_{3}\in \limfunc{Sym}\left( \mathbb{R}^{2}\right) \subset 
\limfunc{End}\left( \mathbb{R}^{2}\right) $, $E_{4}\in \limfunc{Sk}\left( 
\mathbb{R}^{2}\right) \subset \limfunc{End}\left( \mathbb{R}^{2}\right) $ be
defined by $E_{1}\left( x\right) =\left\langle e^{\ast \,1},x\right\rangle
\,e_{1}$, $E_{2}\left( x\right) =\left\langle e^{\ast \,2},x\right\rangle
\,e_{2}$, $E_{3}\left( x\right) =\left\langle e^{\ast \,1},x\right\rangle
\,e_{2}+\left\langle e^{\ast \,2},x\right\rangle \,e_{1}$, $E_{4}\left(
x\right) =\left\langle e^{\ast \,1},x\right\rangle \,e_{2}-\left\langle
e^{\ast \,2},x\right\rangle \,e_{1}$. Observe that $\limfunc{Pf}\left( \left[
E_{1},E_{3}\right] \right) =\limfunc{Pf}\left( E_{4}\right) =1\neq 0$. Thus $%
\Delta _{o\,\#}\left( \left[ \widetilde{y}_{2}\right] \right) \in \mathsf{H}%
^{2}\left( A\right) $ is a nontrivial secondary characteristic class for%
\emph{\ }$\left( TM\oplus \limfunc{End}\left( \mathbb{R}^{2}\right)
,TM\oplus \limfunc{Sk}\left( \mathbb{R}^{2}\right) ,\limfunc{id}\right) $ of
even rank.
\end{proof}


\begin{thebibliography}{99}
\bibitem{B-K-W} B. Balcerzak, J. Kubarski, W. Walas, Primary characteristic
homomorphism of pairs of Lie algebroids and Mackenzie algebroid, In: Lie
Algebroids and Related Topics in Differential Geometry,\ Banach Center Publ.
54 (2001) 135--173.

\bibitem{Chern-Simons} S.-S. Chern, J. Simons, Characteristic Forms and
Geometric Invariants, Ann. of Math. (2) 99 (1974) 48--69.

\bibitem{Chev-Eil} C. Chevalley, S. Eilenberg, Cohomology theory of Lie
groups and Lie algebras, Trans. Amer. Math. Soc. 63 (1948) 85--124.

\bibitem{Cr-1} M. Crainic,\ Differentiable and algebroid cohomology, Van Est
isomorphisms, and characteristic classes, Comment. Math. Helv. 78 (2003)
681--721.

\bibitem{Cr-2} M. Crainic, Chern characters via connections up to homotopy,
Preprint arXiv:math/0009229v2, 2000.

\bibitem{Cr-F} M. Crainic, R. L. Fernandes, Secondary Characteristic Classes
of Lie Algebroids, In: Quantum field theory and noncommutative geometry,
Based on the workshop, Sendai, Japan, November 2002, Lecture Notes in Phys.
662 (2005), Springer, Berlin, 157--176.

\bibitem{Dufour} J.-P. Dufour, Normal forms for Lie algebroids, In: Lie
Algebroids and Related Topics in Differential Geometry, Banach Center Publ.
54 (2001) 35--41.

\bibitem{F} R. L. Fernandes, Lie Algebroids, Holonomy and Characteristic
Classes, Adv. in Math. 170 (2002) 119--179.

\bibitem{Godbillon} C. Godbillon, Cohomologies d'alg\`{e}bres de Lie de
champs de vecteurs formels, S\'{e}minaire Bourbaki, 25e ann\'{e}e, 1972/73,\
n$^{\circ }$\ 421, Lecture Notes in Math. 383 (1974), 69--87.

\bibitem{GHV} W. Greub, S. Halperin, R. Vanstone, Connections, curvature,
and cohomology. Volume III: Cohomology of principal bundles and homogeneous
spaces, Pure and Appl. Math. 47-III, Academic Press, New York, 1976.

\bibitem{Heitsch-Lawson} J. Heitsch, B. Lawson, Transgressions, Chern-Simons
invariants and the classical groups, J. Differential Geom. 9 (1974) 423--434.

\bibitem{He} J.-C. Herz, Pseudo-alg\`{e}bres de Lie. I, II.,\ C. R. Acad.
Sci., Paris 236 (1953) 1935--1937 and 2289--2291.

\bibitem{K-T1} F. Kamber, Ph. Tondeur,\ Alg\`{e}bres de Weil
semi-simpliciales,\ C.R. Acad. Sci. Paris S\'{e}r. A 276 (1973), 1177--1179;
Homomorphisme caract\'{e}ristique d'un fibr\'{e} principal feuillet\'{e}.
Ibid. 276 (1973), 1407--1410; Classes caract\'{e}ristiques d\'{e}riv\'{e}es
d'un fibr\'{e} principal feuillet\'{e}. Ibid. 276 (1973), 1449--1452.

\bibitem{K-T2} F. Kamber, Ph. Tondeur, Characteristic invariants of foliated
bundles,\ Manuscripta Math. 11 (1974) 51--89.

\bibitem{K-T3} F. Kamber, Ph. Tondeur, Foliated Bundles and Characteristic
Classes,\ Lectures Notes in Math. 493, Berlin -- Heidelberg -- New York,
Springer-Verlag, 1975.

\bibitem{K1} J. Kubarski, The Chern-Weil homomorphism of regular Lie
algebroids,\ Publ. D\'{e}p. Math. Nouvelle S\'{e}r. A, Univ. Claude-Dernard,
Lyon, 1991, 1--69.

\bibitem{K3} J. Kubarski, Tangential Chern-Weil homomorphism,\ Proceedings
of Geometric Study of Foliations, Tokyo, November 1993,\ ed. by T. Mizutani,
World Scientific, Singapore, 1994, 327--344.

\bibitem{K5} J. Kubarski, Algebroid nature of the characteristic classes of
flat bundles, In:\ Homotopy and Geometry, Banach Center Publ. 45 (1998)
199--224.

\bibitem{K6} J. Kubarski,\ The Weil algebra and the secondary characteristic
homomorphism of regular Lie algebroids, In: Lie Algebroids and Related
Topics in Differential Geometry,\ Banach Center Publ. 54 (2001) 135--173.

\bibitem{M2} K. Mackenzie, Lie algebroids and Lie pseudoalgebras,\ Bull.
London Math. Soc. 27\ (1995) 97--147.

\bibitem{M1} K. C. H. Mackenzie, General Theory of Lie Groupoids and Lie
Algebroids,\ London Math. Soc. Lecture Note Ser. 213, Cambridge Univ. Press,
2005.

\bibitem{P2} J. Pradines, How to define the graph of a singular foliation,
Cah. Topol. G\'{e}om. Diff\'{e}r. Cat\'{e}g. 26 (1985) 339--380.

\bibitem{Teleman} N. Teleman, A characteristic ring of a Lie algebra
extension, Accad. Naz. Lincei. Rend. Cl. Sci. Fis. Mat. Natur. (8) 52 (1972)
498--506 and ibid. 708--711.

\bibitem{V} I. Vaisman, Characteristic Classes of Lie Algebroid Morphisms,
Differential Geom. Appl. 28 (2010) 635--647.
\end{thebibliography}
\end{document}